\documentclass[12pt]{amsart}
\usepackage{amsmath,amsthm,amssymb}
\newtheorem{theorem}{Theorem}[section]
\newtheorem{corollary}[theorem]{Corollary}
\newtheorem{proposition}[theorem]{Proposition}
\newtheorem{lemma}[theorem]{Lemma}
\theoremstyle{definition}

\theoremstyle{remark}
\newtheorem{remark}[theorem]{Remark}
\numberwithin{equation}{section}
\DeclareMathOperator{\Pic}{Pic}

\DeclareMathOperator{\Grass}{Grass}
\DeclareMathOperator{\Bs}{Bs}
\DeclareMathOperator{\Ram}{Ram}
\DeclareMathOperator{\Image}{Image}

\title[Prym-Torelli theorem for double coverings of elliptic curves]
{Global Prym-Torelli theorem for double coverings of elliptic curves}
\author{Atsushi Ikeda}
\address{Department of Mathematics, School of Engineering,
Tokyo Denki University, Adachi-ku,
Tokyo 120-8551, Japan}
\email{atsushi@mail.dendai.ac.jp}
\begin{document}
\begin{abstract}
 The Prym variety for a branched double covering of a nonsingular projective curve is defined as a polarized abelian variety.
 We prove that any double covering of an elliptic curve which has more than $4$ branch points is recovered from its Prym variety.
\end{abstract}
\maketitle
\section{Introduction}
Let $C$ and $C'$ be nonsingular projective curves, and let $\phi:C\rightarrow{C'}$ be a double covering branched at $2n$ points.
In \cite{Mu2} the Prym variety $P(\phi)$ for the double covering $\phi$ is defined as a polarized abelian variety of dimension $d=g'-1+n$, where $g'$ is the genus of $C'$.
Let $\mathcal{R}=\mathcal{R}_{g',2n}$ be the moduli space of such coverings, and let $\mathcal{A}=\mathcal{A}_{d}$ be the moduli space of polarized abelian varieties of dimension $d$.
Then the construction of the Prym variety defines the Prym map $P:\mathcal{R}\rightarrow\mathcal{A}$, and the Prym-Torelli problem asks whether the Prym map is injective.
If $g'=0$, then it is injective by the classical Torelli theorem for hyperelliptic curves.
We consider the case $g'>0$ and $\dim{\mathcal{R}}\leq\dim{\mathcal{A}}$, where we note that
$\dim{\mathcal{R}}=3g'-3+2n$ and $\dim{\mathcal{A}}=\frac{(g'-1+n)(g'+n)}{2}$.
The generically injectivity for the Prym map has been proved in most cases.
\begin{theorem}\label{162031_29Nov18}
 The Prym map is generically injective in the following cases;
 \begin{enumerate}
  \item \label{120124_5Dec18}
	{\rm (Friedman and Smith \cite{FS}, Kanev \cite{K})}\
	$n=0$ and $\dim{\mathcal{R}}<\dim{\mathcal{A}}$,
  \item \label{162058_29Nov18}
        {\rm (Marcucci and Pirola \cite{MP})}\
	$g'>1$, $n>0$ and $\dim{\mathcal{R}}<\dim{\mathcal{A}}-1$,
  \item \label{114811_18Dec18}
	{\rm (Naranjo and Ortega \cite{NO})}
	$g'>1$, $n>0$ and $\dim{\mathcal{R}}=\dim{\mathcal{A}}-1$,
  \item \label{164128_29Nov18}
	{\rm (Marcucci and Naranjo \cite{MN})}\
	$g'=1$, $n>0$ and $\dim{\mathcal{R}}\leq\dim{\mathcal{A}}$.
 \end{enumerate}
\end{theorem}
The Prym varieties for unramified coverings have been intensively studied because they are principally polarized abelian varieties.
For ramified coverings, Nagaraj and Ramanan \cite{NR} proved the above Theorem~\ref{162031_29Nov18}~\eqref{162058_29Nov18} for $n=2$, and then Marcucci and Pirola \cite{MP} proved it for any $n>0$.
When $g'>1$ and $\dim{\mathcal{R}}=\dim{\mathcal{A}}$, there are only two cases $(g',n)=(6,0),(3,2)$.
If $(g',n)=(6,0)$ then the Prym map is generically finite of degree $27$ (\cite{DS}), and if $(g',n)=(3,2)$ then it is generically finite of degree $3$ (\cite{NR}, \cite{BCV}).\par
Although the Prym map is not injective for many cases in Theorem~\ref{162031_29Nov18} (\cite{D}, \cite{N}, \cite{NR}, \cite{V}),
we prove the injectivity when $g'=1$.
The following is the main result of this paper, which improves Theorem~\ref{162031_29Nov18}\,\eqref{164128_29Nov18}.
\begin{theorem}[Theorem~\ref{155025_28Oct18}]\label{145827_29Nov18}
 If $g'=1$, $n>0$ and $\dim{\mathcal{R}}\leq\dim{\mathcal{A}}$, then the Prym map is injective.
\end{theorem}
To prove this theorem we use the Gauss map for the polarization divisor, which is a standard approach to Torelli problems.
Let $\mathcal{L}$ be an ample invertible sheaf which represents the polarization of the Prym variety $P=P(\phi)$.
For a member $D\in|\mathcal{L}|$, we consider the Gauss map
$$
\Psi_{D}:\,
D\setminus{D_{\mathrm{sing}}}\longrightarrow\mathbf{P}^{d-1}=\Grass{(d-1,H^{0}(P,\Omega_{P}^{1})^{\vee})}.
$$
It is not difficult to show that there exists a member $D_{0}\in|\mathcal{L}|$ such that the branch divisor of $\Psi_{D_{0}}$ recovers the original covering $\phi:C\rightarrow{C'}$ in a similar way as Andreotti's proof \cite{A} of Torelli theorem for hyperelliptic curves.
The essential part of our proof is to distinguish the special member $D_{0}\in|\mathcal{L}|$.
We study the restriction
$\Psi_{D}\vert_{\Bs{|\mathcal{L}|}}:\,
\Bs{|\mathcal{L}|}\setminus{D_{\mathrm{sing}}}\rightarrow\mathbf{P}^{d-1}$
of the Gauss map to the base locus of the linear system $|\mathcal{L}|$.
Although $\Psi_{D}$ is difficult to compute, the restriction $\Psi_{D}\vert_{\Bs{|\mathcal{L}|}}$ is rather simple for any member $D\in|\mathcal{L}|$.
By using the image of $\Psi_{D}\vert_{\Bs{|\mathcal{L}|}}$ and the branch divisor of $\Psi_{D}\vert_{\Bs{|\mathcal{L}|}}$, we can specify the member $D_{0}\in|\mathcal{L}|$ which has the desired property.\par
In Section~\ref{155149_28Oct18}, we summarize some basic properties of bielliptic curves and their Prym varieties.
In Section~\ref{145716_29Nov18}, we explain the strategy of the proof of Theorem~\ref{145827_29Nov18} by using the key Propositions in Section~\ref{150759_29Nov18}.
In Section~\ref{154705_28Oct18}, we explicitly describe the base locus of the linear system of polarization divisors.
In Section~\ref{154736_28Oct18}, we show that the restricted Gauss map $\Psi_{D}\vert_{\Bs{|\mathcal{L}|}}$ is the same map as the restriction of the Gauss map for the theta divisor on Jacobian variety of $C$.
By giving a simple description for $\Psi_{D}\vert_{\Bs{|\mathcal{L}|}}$, we prove some properties on the branch divisor of $\Psi_{D}\vert_{\Bs{|\mathcal{L}|}}$.
In Section~\ref{150759_29Nov18}, we present key Propositions, which are consequences of the results in Section~\ref{154736_28Oct18}.\par
In this paper, we work over an algebraically closed field $k$ of characteristic $\neq{2}$.
\section{Properties of bielliptic curves and Prym varieties}\label{155149_28Oct18}
Let $C$ be a nonsingular projective curve of genus $g$ over $k$, and let $\sigma$ be an involution on $C$.
In this paper, we call the pair $(C,\sigma)$ a bielliptic curve of genus $g$, if $g>1$ and the quotient $E=C/\sigma$ is a nonsingular curve of genus $1$.
We denote by $\phi:C\rightarrow{E}$ the quotient morphism.
First we note the following.
\begin{lemma}[\cite{N} (3.3)]\label{123518_3Aug18}
 Let $(C,\sigma)$ be a bielliptic curve of genus $g$.
 If $g>3$, then $C$ is not a hyperelliptic curve.
\end{lemma}
Let $N:J(C)\rightarrow{J(E)}$ be the norm map of $\phi$, which is a homomorphism on their Jacobian varieties.
\begin{lemma}[\cite{Mu2}]\label{172616_1Dec18}
 Let $\phi:C\rightarrow{E}$ be the covering defined from a bielliptic curve $(C,\sigma)$.
 \begin{enumerate}
  \item $\phi^{*}:\Pic^{0}{(E)}\rightarrow\Pic^{0}{(C)}$ is injective.
  \item The kernel $P$ of the norm map $N:J(C)\rightarrow{J(E)}$ is reduced and connected.
 \end{enumerate}
\end{lemma}
By Lemma~\ref{172616_1Dec18}, the kernel $P$ of the norm map $N$ is an abelian variety of dimension $n=g-1$.
Let $P^{\vee}$ be the dual abelian variety of $P$, and let $\lambda_{P}:P\rightarrow{P^{\vee}}$ be the polarization isogeny which is defined as the restriction of the principal polarization on the Jacobian variety $J(C)$.
Then the polarized abelian variety $(P,\lambda_{P})$ is called the Prym variety for the covering $\phi:C\rightarrow{E}$.
We denote by $K(P)\subset{P}$ the kernel of the polarization $\lambda_{P}:P\rightarrow{P^{\vee}}$.
An ample invertible sheaf $\mathcal{L}$ on $P$ represents the the polarization $\lambda_{P}$, if the polarization isogeny $\lambda_{P}$ is given by
$$
\lambda_{P}:\,P(k)\longrightarrow{P^{\vee}(k)=\Pic^{0}{(P)}};\,
x\longmapsto{t_{x}^{*}\mathcal{L}\otimes\mathcal{L}^{\vee}},
$$
where $t_{x}:P\rightarrow{P}$ denotes the translation by $x\in{P(k)}$.
\begin{lemma}[\cite{Mu2}]\label{172812_1Dec18}
 Let $(P,\lambda_{P})$ be the Prym variety defined from a bielliptic curve $(C,\sigma)$, and let $\mathcal{L}$ be an ample invertible sheaf which represents $\lambda_{P}$.
 \begin{enumerate}
  \item $K(P)=\phi^{*}J(E)_{2}\subset{J(C)}$,
 where $J(E)_{2}$ denotes the set of points of order $2$ on $J(E)$.
  \item $\deg{\lambda_{P}}=4$ and $h^{0}(P,\mathcal{L})=2$.
 \end{enumerate}
\end{lemma}
\section{Proof of Main theorem}\label{145716_29Nov18}
The main result of this paper is the following.
\begin{theorem}\label{155025_28Oct18}
 If $g>3$, then the isomorphism class of a bielliptic curve of genus $g$ is determined by the isomorphism class of its Prym variety.
\end{theorem}
Let $(P,\lambda_{P})$ be the Prym variety of dimension $n\geq3$ defined from a bielliptic curve $(C,\sigma)$ of genus $g=n+1$.
We will recover the data $(E,e_{1}+\dots+e_{2n},\eta)$ from the polarized abelian variety $(P,\lambda_{P})$, where $E=C/\sigma$ is the quotient curve, $e_{1}+\dots+e_{2n}$ is the branch divisor of the covering $\phi:C\rightarrow{E}$, and $\eta\in\Pic{(E)}$ is the invertible sheaf with $\phi^{*}\eta\cong\Omega_{C}^{1}$.
We remark that $\eta^{\otimes2}\cong\mathcal{O}_{E}(e_{1}+\dots+e_{2n})$, and $\eta$ is the invertible sheaf which determines the double covering with the branch divisor $e_{1}+\dots+e_{2n}$.
\begin{proof}[Proof of Theorem~\ref{155025_28Oct18}]
 Let $\mathcal{L}$ be an ample invertible sheaf on $P$ which represents the the polarization $\lambda_{P}$.
 We denote by $K(P)$ the Kernel of $\lambda_{P}:P\rightarrow{P}^{\vee}$.
 By Lemma~\ref{172812_1Dec18}, we have $\sharp{K(P)}=4$ and $h^{0}(P,\mathcal{L})=2$.
 We define the subset $\Pi_{\mathcal{L}}$ in the linear pencil $|\mathcal{L}|$ by
 $$
 \Pi_{\mathcal{L}}=
 \{D\in|\mathcal{L}|\mid{t_{x}(D)=D\subset{P}}\
 \text{for some}\
 x\in{K(P)\setminus\{0\}}\},
 $$
 where $t_{x}$ is the translation by $x\in{P(k)}$.
 By Lemma~\ref{142326_6Oct18}, $\Pi_{\mathcal{L}}$ is a set of $6$ members for any representative $\mathcal{L}$ of the polarization $\lambda_{P}$.
 For a member $D\in|\mathcal{L}|\setminus\Pi_{\mathcal{L}}$, we consider the Gauss map
 $$
 \Psi_{D}:D\setminus{D_{\mathrm{sing}}}\longrightarrow
 \mathbf{P}^{n-1}=\Grass{(n-1,H^{0}(P,\Omega_{P}^{1})^{\vee})},
 $$
 where $\Psi_{D}(x)$ is defined by the inclusion
 $T_{x}(D)\subset{T_{x}(P)}\cong{H^{0}(P,\Omega_{P}^{1})^{\vee}}$
 of the tangent spaces at the point $x\in{D\setminus{D_{\mathrm{sing}}}}\subset{P}$.
 We set $U_{D}=\Bs|\mathcal{L}|\setminus{D_{\mathrm{sing}}}$, where $\Bs|\mathcal{L}|\subset{P}$ denotes the set of base points of the pencil $|\mathcal{L}|$.
 Let $X'_{D}=\overline{\Psi_{D}(U_{D})}\subset\mathbf{P}^{n-1}$ be the Zariski closure of $\Psi_{D}(U_{D})\subset\mathbf{P}^{n-1}$, and let $\nu_{D}:X_{D}\rightarrow{X'_{D}}$ be the normalization.
 By Lemma~\ref{152342_24Nov18}, $U_{D}$ is a nonsingular variety, hence there is a unique morphism $\psi_{D}:U_{D}\rightarrow{X_{D}}$ such that $\Psi_{D}\vert_{U_{D}}=\nu_{D}\circ\psi_{D}$.
 We consider the closed subset $Z_{D}=\overline{\psi_{D}(\Ram{(\psi_{D})})}\subset{X_{D}}$, where $\Ram{(\psi_{D})}\subset{U_{D}}$ denotes the ramification divisor of $\psi_{D}$.
 By Proposition~\ref{174153_24Nov18}, $Z_{D}$ has a canonical decomposition $Z_{D}=\bigcup_{i=1}^{2n}Z_{D,i}$, and there is a unique hyperplane $H_{D,i}\subset\mathbf{P}^{n-1}$ such that $\nu_{D}(Z_{D,i})\subset{H_{D,i}}$ for any $1\leq{i}\leq{2n}$.
 Then the effective divisor $\nu_{D}^{*}{H_{D,i}}-Z_{D,i}$ on $X_{D}$ has $2$ irreducible components for general $D\in|\mathcal{L}|\setminus\Pi_{\mathcal{L}}$, and these components coincide for special $D\in|\mathcal{L}|\setminus\Pi_{\mathcal{L}}$.
 We define the subset $\Pi'_{\mathcal{L}}$ in the linear pencil $|\mathcal{L}|$ by
 $$
 \Pi'_{\mathcal{L}}=\{D\in|\mathcal{L}|\setminus\Pi_{\mathcal{L}}\mid
 \text{$\nu_{D}^{*}{H_{D,i}}-Z_{D,i}$ is irreducible for $1\leq{i}\leq{2n}$}\}.
 $$
 By Lemma~\ref{134511_15Dec18}, $\Pi'_{\mathcal{L}}$ is a set of $4$ members for any representative $\mathcal{L}$ of the polarization $\lambda_{P}$.
 For a member $D\in\Pi'_{\mathcal{L}}$, we consider the dual variety $(X'_{D})^{\vee}\subset(\mathbf{P}^{n-1})^{\vee}$ of $X'_{D}\subset\mathbf{P}^{n-1}$ and the dual variety $H_{D,i}^{\vee}\subset(\mathbf{P}^{n-1})^{\vee}$ of $H_{D,i}\subset\mathbf{P}^{n-1}$.
 By Proposition~\ref{173842_24Nov18}, $H_{D,i}^{\vee}$ is a point on $(X'_{D})^{\vee}$, and we have an isomorphism
 $$
 (E,e_{1}+\dots+e_{2n},\eta)\cong
 ((X'_{D})^{\vee},H_{D,1}^{\vee}+\dots+H_{D,2n}^{\vee},\mathcal{O}_{(\mathbf{P}^{n-1})^{\vee}}(1)\vert_{(X'_{D})^{\vee}}).
 $$
\end{proof}
\section{Pencil of polarization divisors}\label{154705_28Oct18}
Let $(C,\sigma)$ be a bielliptic curve of genus $g=n+1>3$.
For $\delta\in\Pic^{n}{(C)}$, we set the divisor $W_{\delta}\subset{J(C)}$ by
$$
W_{\delta}(k)=\{L\in\Pic^{0}{(C)}=J(C)(k)\mid
h^{0}(C,L\otimes\delta)>0\}.
$$
We remark that the singular locus of $W_{\delta}$ is given by
$$
W_{\delta,\mathrm{sing}}(k)=\{L\in\Pic^{0}{(C)}\mid
h^{0}(C,L\otimes\delta)>1\},
$$
and $\dim{W_{\delta,\mathrm{sing}}}=n-3$ (\cite[Proposition~8]{AM}), because $C$ is not a hyperelliptic curve by Lemma~\ref{123518_3Aug18}.
Let $\lambda_{C}:J(C)\rightarrow{J(C)^{\vee}}$ be the homomorphism defined by
$$
\lambda_{C}:\,J(C)(k)\rightarrow{J(C)^{\vee}(k)}=\Pic^{0}{(J(C))};\,
x\longmapsto{[t_{x}^{*}\mathcal{O}_{C}(W_{\delta})\otimes\mathcal{O}_{C}(-W_{\delta})]},
$$
which does not depend on the choice of $\delta\in\Pic^{n}{(C)}$.
Let $\iota_{q}:C\rightarrow{J(C)}$ be the morphism defined by
$$
\iota_{q}:\,
C(k)\longrightarrow\Pic^{0}{(C)}={J(C)(k)};\,
q'\longmapsto[\mathcal{O}_{C}(q'-q)].
$$
for $q\in{C(k)}$.
\begin{lemma}\label{115529_3Aug18}
 $$
 x=\iota_{q}^{*}[\mathcal{O}_{J(C)}(W_{\delta}-W_{\delta+x})]
 \in\Pic^{0}{(C)}
 $$
 for any $q\in{C(k)}$ and $x\in\Pic^{0}{(C)}$.
\end{lemma}
\begin{proof}
 The statement means that $(-1)\circ\lambda_{C}$ is the inverse of the homomorphism $\iota_{q}^{*}:J(C)^{\vee}\rightarrow{J(C)}$ defined by the pull-buck $\iota_{q}^{*}:\Pic^{0}{(J(C))}\rightarrow\Pic^{0}{(C)}$ of invertible sheaves.
 It is well-known (\cite[Lemma~6.9]{Mi}).
\end{proof}
Let $P$ be the kernel of the Norm map $N:J(C)\rightarrow{J(E)}$, and let $D_{\delta}\subset{P}$ the fiber of the restriction of the norm map $N\vert_{W_{\delta}}:W_{\delta}\rightarrow{J(E)}$ at $0\in{J(E)}$.
We denote by $\mathcal{L}_{\delta}=\mathcal{O}_{P}(D_{\delta})=\mathcal{O}_{J(C)}(W_{\delta})\vert_{P}$ the restriction of $\mathcal{O}_{J(C)}(W_{\delta})$ to $P$.
Since $W_{\delta}$ is the theta divisor of $J(C)$, the ample invertible sheaf $\mathcal{L}_{\delta}$ represents the the polarization $\lambda_{P}$.
\begin{lemma}\label{165029_6Aug18}
 $D_{\delta+\phi^{*}s}\subset{P}$ is a member of the linear system $|\mathcal{L}_{\delta}|$ for any $s\in\Pic^{0}{(E)}$.
\end{lemma}
\begin{proof}
 By Lemma~\ref{115529_3Aug18},
 $$
 \phi^{*}s=\iota_{q}^{*}[\mathcal{O}_{J(C)}(W_{\delta}-W_{\delta+\phi^{*}s})]
 \in\Pic^{0}{(C)}
 $$
 for $s\in\Pic^{0}{(E)}$ and $q\in{C(k)}$.
 We set $s'\in\Pic^{0}{(J(E))}$ by $s=\iota_{\phi(q)}^{*}s'$, where $\iota_{\phi(q)}:E\overset{\sim}{\rightarrow}{J(E)}$ is the isomorphism determined by $\iota_{\phi(q)}(\phi(q))=0$.
 Then we have
 $$
 N^{*}s'=[\mathcal{O}_{J(C)}(W_{\delta}-W_{\delta+\phi^{*}s})]
 \in\Pic^{0}{(J(C))},
 $$
 because $\phi^{*}s=\iota_{q}^{*}N^{*}s'$ and $\iota_{q}^{*}:\Pic^{0}(J(C))\rightarrow\Pic^{0}{(C)}$ is an isomorphism.
 Since $(N^{*}s')\vert_{P}=0\in\Pic^{0}{(P)}$, we have
 $$
 \mathcal{O}_{P}(D_{\delta+\phi^{*}s})
 \cong\mathcal{O}_{J(C)}(W_{\delta+\phi^{*}s})\vert_{P}
 \cong\mathcal{O}_{J(C)}(W_{\delta})\vert_{P}
 =\mathcal{L}_{\delta}.
 $$
\end{proof}
We denote by $C^{(i)}$ the $i$-th symmetric products of $C$.
For $\delta\in\Pic^{n}{(C)}$, we define the morphism $\beta_{\delta}^{i}:C^{(n-2i)}\times{E^{(i)}}\rightarrow{J(C)}$ by
\begin{align*}
 \beta^{i}_{\delta}:\,C^{(n-2i)}(k)\times{E^{(i)}(k)}&\longrightarrow
 J(C)(k)=\Pic^{0}{(C)};\\
 (q_{1}+\dots+q_{n-2i},p_{1}+\dots+p_{i})&\longmapsto
 \mathcal{O}_{C}\Bigl(\sum_{j=1}^{n-2i}q_{j}\Bigr)\otimes
 \phi^{*}\mathcal{O}_{E}\Bigl(\sum_{j=1}^{i}p_{j}\Bigr)\otimes\delta^{\vee}.
\end{align*}
We remark that $W_{\delta}=\Image{(\beta_{\delta}^{0})}$, and we set
$$
B^{i}_{\delta}=
\begin{cases}
 \Image{(\beta_{\delta}^{i})}&(1\leq{2i}\leq{n}),\\
 \emptyset&(2i>n).
\end{cases}
$$
\begin{lemma}\label{171048_17Sep18}
 $B^{1}_{\delta}\setminus{W_{\delta,\mathrm{sing}}}\neq\emptyset$ and $B_{\delta}^{2}\subset{W_{\delta,\mathrm{sing}}}$.
\end{lemma}
\begin{proof}
 Let $B^{1}$ be the image of the morphism $\beta^{1}:C^{(n-2)}\times{E}\rightarrow{C^{(n)}}$ defined by
 $$
 \beta^{1}:\,
 C^{(n-2)}(k)\times{E(k)}\rightarrow{C^{(n)}(k)};\,
 (q_{1}+\dots+q_{n-2},\phi(q))\longmapsto
 q_{1}+\dots+q_{n-2}+q+\sigma(q).
 $$
 Since $C$ is not a hyperelliptic curve, we have $\dim{(\beta_{\delta}^{0})^{-1}(W_{\delta,\mathrm{sing}})}=n-2<n-1=\dim{B^{1}}$, hence $B_{\delta}^{1}=\beta_{\delta}^{0}(B^{1})\nsubseteq{W_{\delta,\mathrm{sing}}}$.\par
 To prove the second statement, we assume that $n\geq4$, because $B_{\delta}^{2}=\emptyset$ for $n=3$.
 Let $F\subset{C^{(n-4)}\times{E^{(2)}}}$ be the fiber of the composition
 $$
 C^{(n-4)}\times{E^{(2)}}\overset{\beta^{2}_{\delta}}{\longrightarrow}
 J(C)\overset{N}{\longrightarrow}J(E)
 $$
 at $\eta-{N(\delta)}\in{J(E)(k)}$, where $\eta\in\Pic^{n}{(E)}$ denotes the unique invertible sheaf on $E$ with $\phi^{*}\eta\cong\Omega_{C}^{1}$.
 We set $U=(C^{(n-4)}\times{E^{(2)}})\setminus(F\cup(C^{(n-4)}\times{\Delta_{E}}))$, where $\Delta_{E}\subset{E^{(2)}}$ denotes the image of the diagonal in $E\times{E}$.
 For $y=(q_{1}+\dots+q_{n-4},\phi(q)+\phi(r))\in{U(k)}$, there are points $q',r'\in{C(k)}$ such that
 $$
 \begin{cases}
  \mathcal{O}_{E}(\phi(q'))\cong
  \eta\otimes\mathcal{O}_{E}
  (-\phi(q_{1})-\dots-\phi(q_{n-4})-\phi(q)-2\phi(r)),\\
  \mathcal{O}_{E}(\phi(r'))\cong
  \eta\otimes\mathcal{O}_{E}
  (-\phi(q_{1})-\dots-\phi(q_{n-4})-2\phi(q)-\phi(r)).
 \end{cases}
 $$
 Then we have
 \begin{align*}
  &\Omega_{C}^{1}(-q_{1}-\dots-q_{n-4}-q-\sigma(q)-r-\sigma(r))\\
  \cong&\phi^{*}\eta\otimes\mathcal{O}_{C}
  (-q_{1}-\dots-q_{n-4}-q-\sigma(q)-r-\sigma(r))\\
  \cong&
  \mathcal{O}_{C}
  (\sigma(q_{1})+\dots+\sigma(q_{n-4})+q'+\sigma(q')+r+\sigma(r))\\
  \cong&
  \mathcal{O}_{C}
  (\sigma(q_{1})+\dots+\sigma(q_{n-4})+q+\sigma(q)+r'+\sigma(r')).
 \end{align*}
 We remark that $\phi(q)\neq\phi(q')$ and $\phi(q)\neq\phi(r)$, because $y\notin{F}$ and $y\notin{C^{(n-4)}\times{\Delta_{E}}}$.
 Hence we have
 $h^{0}(C,\Omega_{C}^{1}(-q_{1}-\dots-q_{n-4}-q-\sigma(q)-r-\sigma(r)))
 >1$
 and $\beta_{\delta}^{2}(y)\in{W_{\delta,\mathrm{sing}}}$.
 Since
 $U\subset(\beta_{\delta}^{2})^{-1}(W_{\delta,\mathrm{sing}})$
 is a dense subset of $C^{(n-4)}\times{E^{(2)}}$, we have
 $C^{(n-4)}\times{E^{(2)}}
 =(\beta_{\delta}^{2})^{-1}(W_{\delta,\mathrm{sing}})$.
\end{proof}
\begin{lemma}\label{184538_11Aug18}
 For $s\in\Pic^{0}{(E)}$, $D_{\delta}=D_{\delta+\phi^{*}s}\subset{P}$ if and only if
 $s=0$ or $s=\eta-N(\delta)$.
\end{lemma}
\begin{proof}
 For $L\in{P(k)}\subset\Pic^{0}{(C)}$, we have $0=\phi^{*}N(L)=L+\sigma^{*}L\in\Pic^{0}{(C)}$.
 Hence we have $D_{\delta}=D_{\delta+\phi^{*}(\eta-N(\delta))}$, because
 $$
 L\in{D}_{\delta}(k)
 \Longleftrightarrow
 h^{0}(C,L\otimes\delta)>0
 \Longleftrightarrow
 h^{0}(C,\sigma^{*}L\otimes\sigma^{*}\delta)>0
 \Longleftrightarrow
 h^{0}(C,L^{\vee}\otimes\sigma^{*}\delta)>0
 $$
 $$
 \Longleftrightarrow
 h^{0}(C,\Omega_{C}^{1}\otimes{L}\otimes\sigma^{*}\delta^{\vee})>0
 \Longleftrightarrow
 L\in{D}_{[\Omega_{C}^{1}]-\sigma^{*}\delta}(k)
 =D_{\delta+\phi^{*}(\eta-N(\delta))}(k).
 $$
 We assume that $D_{\delta}=D_{\delta+\phi^{*}s}$ for $s\neq0\in\Pic^{0}{(E)}$.
 Let $\alpha_{\delta+\phi^{*}s}:C^{(n-2)}\times{C}\rightarrow{J(C)}$ be the morphism defined by
 \begin{align*}
  \alpha_{\delta+\phi^{*}s}:\,C^{(n-2)}(k)\times{C(k)}&\longrightarrow\Pic^{0}{(C)}={J(C)(k)};\\
  (q_{1}+\dots+q_{n-2},q)&\longmapsto
  \mathcal{O}_{C}(q_{1}+\dots+q_{n-2}+2q)\otimes\delta^{\vee}\otimes\phi^{*}s^{\vee}.
 \end{align*}
 Then the set
 $D_{\delta}\setminus{(W_{\delta,\mathrm{sing}}\cup{B^{1}_{\delta}}\cup\Image{(\alpha_{\delta+\phi^{*}s})})}$
 is not empty, because
 $$
 \dim{D_{\delta}\cap{(W_{\delta,\mathrm{sing}}\cup{B^{1}_{\delta}}\cup\Image{(\alpha_{\delta+\phi^{*}s})})}}
 <n-1=\dim{D_{\delta}}.
 $$
 For
 $L\in
 D_{\delta}(k)\setminus{(W_{\delta,\mathrm{sing}}\cup{B^{1}_{\delta}}\cup\Image{(\alpha_{\delta+\phi^{*}s})})}$,
 there is $r_{1}+\dots+r_{n}\in{C^{(n)}}(k)$ such that
 $L\otimes\delta\otimes\phi^{*}s\cong\mathcal{O}_{C}(r_{1}+\dots+r_{n})$,
 because $L\in{D_{\delta}(k)}=D_{\delta+\phi^{*}s}(k)\subset{W_{\delta+\phi^{*}s}(k)}$.
 Since $L\in{W_{\delta}(k)\setminus{W_{\delta,\mathrm{sing}}(k)}}$, we have
 $h^{0}(C,\Omega_{C}^{1}\otimes{L^{\vee}}\otimes\delta^{\vee})
 =h^{0}(C,L\otimes\delta)=1$.
 Let $q_{1}+\dots+q_{n}\in{C^{(n)}(k)}$ and $q'_{1}+\dots+q'_{n}\in{C^{(n)}(k)}$ be the effective divisors defined by
 $$
 L\otimes\delta\cong\mathcal{O}_{C}(q_{1}+\dots+q_{n}),\quad
 \Omega_{C}^{1}\otimes{L^{\vee}}\otimes\delta^{\vee}\cong\mathcal{O}_{C}(q'_{1}+\dots+q'_{n}).
 $$
 Let $\phi(u_{i})\in{E(k)}$ be the point determined by $s=[\mathcal{O}_{E}(\phi(r_{i})-\phi(u_{i}))]$.
 Then
 $$
 L\otimes\delta\otimes\mathcal{O}_{C}(\sigma(r_{i}))
 \cong\mathcal{O}_{C}(r_{1}+\dots+r_{n}-r_{i}+u_{i}+\sigma(u_{i})).
 $$
 If $\sigma(r_{i})\notin\{q'_{1},\dots,q'_{n}\}$, then
 $$
 h^{0}(C,L\otimes\delta\otimes\mathcal{O}_{C}(\sigma(r_{i})))=
 h^{0}(C,\Omega_{C}^{1}\otimes{L^{\vee}}\otimes\delta^{\vee}
 \otimes\mathcal{O}_{C}(-\sigma(r_{i})))+1=1,
 $$
 hence
 $$
 q_{1}+\dots+q_{n}+\sigma(r_{i})=
 r_{1}+\dots+r_{n}-r_{i}+u_{i}+\sigma(u_{i}).
 $$
 Since $s\neq0$, we have $\sigma(r_{i})=r_{j}$ for some $j\neq{i}$, and $L\in{B^{1}_{\delta}(k)}$.
 It is a contradiction to $L\notin{B^{1}_{\delta}(k)}$,
 hence $\sigma(r_{i})\in\{q'_{1},\dots,q'_{n}\}$ for any $1\leq{i}\leq{n}$.
 Here the condition $L\notin\Image{(\alpha_{\delta+\phi^{*}s})}$ implies that $\sharp\{r_{1},\dots,r_{n}\}=n$ and
 $$
 {L^{\vee}}\otimes\sigma^{*}\delta\otimes\phi^{*}s
 \cong\mathcal{O}_{C}(\sigma(r_{1})+\dots+\sigma(r_{n}))
 =\mathcal{O}_{C}(q'_{1}+\dots+q'_{n})
 \cong\Omega_{C}^{1}\otimes{L^{\vee}}\otimes\delta^{\vee}.
 $$
 Hence we have
 $\phi^{*}s=[\Omega_{C}^{1}]-\delta-\sigma^{*}\delta
 =\phi^{*}(\eta-N(\delta))$, and $s=\eta-N(\delta)$ by Lemma~\ref{172616_1Dec18}.
\end{proof}
Let $B_{\delta}\subset{J(C)}$ be the subset
$$
B_{\delta}=\bigcap_{s\in\Pic^{0}{(E)}}W_{\delta+\phi^{*}s}.
$$
\begin{lemma}\label{154952_6Aug18}
 $B_{\delta}\setminus{W_{\delta,\mathrm{sing}}}={B^{1}_{\delta}}\setminus{W_{\delta,\mathrm{sing}}}$.
\end{lemma}
\begin{proof}
 If $L\in{B^{1}_{\delta}(k)}$, then
 $L\otimes\delta\cong\mathcal{O}_{C}(q_{1}+\dots+q_{n-2}+q+\sigma(q))$
 for some $q_{1},\dots,q_{n-2},q\in{C(k)}$.
 For $s\in\Pic^{0}{(E)}$, there is a point $q'\in{C(k)}$ such that $s=[\mathcal{O}_{E}(\phi(q')-\phi(q))]$.
 Since
 $L\otimes\delta\otimes\phi^{*}s
 \cong\mathcal{O}_{C}(q_{1}+\dots+q_{n-2}+q'+\sigma(q'))$,
 we have $h^{0}(C,L\otimes\delta\otimes\phi^{*}s)>0$ and $L\in{W_{\delta+\phi^{*}s}(k)}$.
 Hence the inclusion $B^{1}_{\delta}\subset{B_{\delta}}$ holds.\par
 For $L\in{B_{\delta}(k)\setminus{W_{\delta,\mathrm{sing}}(k)}}$, there is a unique $r_{1}+\dots+r_{n}\in{C^{(n)}(k)}$ such that $\Omega_{C}^{1}\otimes{L^{\vee}}\otimes\delta^{\vee}\cong\mathcal{O}_{C}(r_{1}+\dots+r_{n})$, because
 $h^{0}(C,\Omega_{C}^{1}\otimes{L^{\vee}}\otimes\delta^{\vee})
 =h^{0}(C,L\otimes\delta)=1$.
 Let $\Sigma\subset\Pic^{0}{(E)}$ be the finite subset defined by
 $$
 \Sigma=\{s\in\Pic^{0}{(E)}\mid{L\otimes\phi^{*}s\in
 \beta_{\delta}^{0}(\Delta^{(n)})}\},
 $$
 where $\Delta=\{\sigma(r_{1}),\dots,\sigma(r_{n})\}\subset{C}$ and $\Delta^{(n)}\subset{C^{(n)}}$.
 For $s\in\Pic^{0}{(E)}\setminus(\Sigma\cup\{0\})$, there is a divisor $q_{1}+\dots+q_{n}\in{C^{(n)}(k)}$ such that
 $L\otimes\delta\otimes\phi^{*}s\cong\mathcal{O}_{C}(q_{1}+\dots+q_{n})$,
 because $L\in{B_{\delta}(k)}\subset{W_{\delta+\phi^{*}s}(k)}$.
 Since $s\notin\Sigma$, we may assume that $q_{n}\notin\Delta$.
 The condition $\sigma(q_{n})\notin\{r_{1},\dots,r_{n}\}$ implies that
 $h^{0}(C,\Omega_{C}^{1}\otimes{L^{\vee}}\otimes
 \delta^{\vee}\otimes\mathcal{O}_{C}(-\sigma(q_{n})))=0$
 and
 $h^{0}(C,L\otimes\delta\otimes\mathcal{O}_{C}(\sigma(q_{n})))
 =1$.
 Let $\phi(q')\in{E(k)}$ be the point determined by $s=[\mathcal{O}_{E}(\phi(q_{n})-\phi(q'))]$.
 Then
 $$
 L\otimes\delta\otimes\mathcal{O}_{C}(\sigma(q_{n}))
 \cong\mathcal{O}_{C}(q_{1}+\dots+q_{n-1}+q'+\sigma(q')).
 $$
 Since $s\neq0$, we have
 $\sigma(q_{n})\in\{q_{1},\dots,q_{n-1}\}$ and $L\in{B^{1}_{\delta}(k)}$.
\end{proof}
\begin{lemma}\label{135618_16Aug18}
 The map
 $$
 \Pic^{0}{(E)}\longrightarrow|\mathcal{L}_{\delta}|;\,
 s\longmapsto{D_{\delta+\phi^{*}s}}
 $$
 is a double covering, and the base locus $\Bs{|\mathcal{L}_{\delta}|}$ of the linear system $|\mathcal{L}_{\delta}|$ is $B_{\delta}\cap{P}$, which is of dimension $n-2$.
\end{lemma}
\begin{proof}
 The map is well-defined by Lemma~\ref{165029_6Aug18}.
 Since $\dim{|\mathcal{L}_{\delta}|}=1$, it is a double covering by Lemma~\ref{184538_11Aug18}.
 Hence we have
 $$
 \Bs{|\mathcal{L}_{\delta}|}=\bigcap_{s\in\Pic^{0}{(E)}}D_{\delta+\phi^{*}s}=B_{\delta}\cap{P}.
 $$
 By Lemma~\ref{171048_17Sep18}, $B^{1}_{\delta}$ is irreducible of dimension $n-1$.
 Since the restriction of the Norm map $N\vert_{B^{1}_{\delta}}:B^{1}_{\delta}\rightarrow{J(E)}$ is surjective, we have $\dim{B^{1}_{\delta}\cap{P}}=n-2$, hence $\dim{B_{\delta}\cap{P}}=n-2$ by Lemma~\ref{154952_6Aug18}.
\end{proof}
\begin{lemma}\label{135426_16Aug18}
 Let $\mathcal{L}$ be an ample invertible sheaf which represents the polarization $\lambda_{P}$ on $P$, then there is $\delta\in\Pic^{n}{(C)}$ such that $N(\delta)=\eta$ and $\mathcal{L}\cong\mathcal{L}_{\delta}$.
\end{lemma}
\begin{proof}
 For any $\delta'\in\Pic^{n}{(C)}$, we have $\mathcal{L}\otimes\mathcal{L}_{\delta'}^{\vee}\in\Pic^{0}{(P)}$, because $\mathcal{L}_{\delta'}$ gives the same polarization as $\lambda_{P}$.
 Then $\mathcal{L}\cong{t_{x}^{*}\mathcal{L}_{\delta'}}\cong\mathcal{L}_{\delta'+x}$ for some $x\in{P(k)}$.
 Let $s\in\Pic^{0}{(E)}$ be a point with $2s=\eta-N(\delta'+x)$.
 For $\delta=\delta'+x+\phi^{*}s$, we have $N(\delta)=\eta$ and $\mathcal{L}\cong\mathcal{L}_{\delta}$
\end{proof}
For an ample invertible sheaf $\mathcal{L}$ which represents the polarization $\lambda_{P}$, we set a subset in the linear system $|\mathcal{L}|$ by
$$
\Pi_{\mathcal{L}}=
\{D\in|\mathcal{L}|\mid{t_{x}(D)=D}\
\text{for some}\
x\in{K(P)\setminus\{0\}}\},
$$
where $K(P)$ is the kernel of the polarization $\lambda_{P}$.
\begin{lemma}\label{142326_6Oct18}
 $\sharp{\Pi_{\mathcal{L}}}=6$.
\end{lemma}
\begin{proof}
 By Lemma~\ref{135426_16Aug18}, there is $\delta\in\Pic^{n}{(C)}$ such that $N(\delta)=\eta$ and $\mathcal{L}\cong\mathcal{L}_{\delta}$.
 For any $D\in|\mathcal{L}_{\delta}|$, by Lemma~\ref{135618_16Aug18}, there is $s\in\Pic^{0}{(E)}$ such that $D=D_{\delta+\phi^{*}s}$.
 If $D_{\delta+\phi^{*}s}\in{\Pi_{\mathcal{L}_{\delta}}}$, then by Lemma~\ref{172812_1Dec18}, there is $t\in{J(E)_{2}}\setminus\{0\}$ such that $t_{\phi^{*}t}(D_{\delta+\phi^{*}s})=D_{\delta+\phi^{*}s}$.
 Since $t_{\phi^{*}t}(D_{\delta+\phi^{*}s})=D_{\delta+\phi^{*}(s-t)}$ and $t\neq0$, by Lemma~\ref{184538_11Aug18}, we have
 $$
 \delta+\phi^{*}(s-t)=\delta+\phi^{*}s+\phi^{*}(\eta-N(\delta+\phi^{*}s))
 =\delta-\phi^{*}s,
 $$
 hence $t=2s$ by Lemma~\ref{172616_1Dec18}.
 It means that
 $$
 \Pi_{\mathcal{L}_{\delta}}=
 \{D_{\delta+\phi^{*}s}\in|\mathcal{L}_{\delta}|\mid
 s\in{J(E)_{4}}\setminus{J(E)_{2}}\}.
 $$
 Since $\sharp(J(E)_{4}\setminus{J(E)_{2}})=12$ and $D_{\delta+\phi^{*}s}=D_{\delta-\phi^{*}s}$, we have $\sharp{\Pi_{\mathcal{L}}}=6$.
\end{proof}
\section{Gauss maps}\label{154736_28Oct18}
\subsection{Gauss map for Jacobian and Gauss map for Prym}
Let
$$
\Psi_{J(C),\delta}:W_{\delta}\setminus{W_{\delta,\mathrm{sing}}}\longrightarrow
\mathbf{P}(H^{0}(C,\Omega_{C}^{1})^{\vee})
=\Grass{(n,H^{0}(C,\Omega_{C}^{1})^{\vee})}
$$
be the Gauss map for the subvariety $W_{\delta}\subset{J(C)}$.
For $L\in{W_{\delta}(k)}\setminus{W_{\delta,\mathrm{sing}}(k)}$, the tangent space $T_{L}(W_{\delta})$ of $W_{\delta}$ at $L$ defines the image $\Psi_{J(C),\delta}(L)$ by the natural identifications
$$
T_{L}(W_{\delta})\subset
T_{L}(J(C))\cong(\Omega_{J(C)}^{1}(L))^{\vee}\cong
H^{0}(J(C),\Omega_{J(C)}^{1})^{\vee}\cong
H^{0}(C,\Omega_{C}^{1})^{\vee}.
$$
\begin{lemma}\label{162111_7Aug18}
 For $L\in{W_{\delta}(k)\setminus{W_{\delta,\mathrm{sing}}(k)}}$, the image
 $\Psi_{J(C),\delta}(L)$ of the Gauss map is identified with the canonical divisor
 $$
 q_{1}+\dots+q_{n}+q'_{1}+\dots+q'_{n}\in|\Omega_{C}^{1}|=
 \Grass{(1,H^{0}(C,\Omega_{C}^{1}))}\cong
 \mathbf{P}(H^{0}(C,\Omega_{C}^{1})^{\vee}),
 $$
 where the effective divisors $q_{1}+\dots+q_{n}$ and $q'_{1}+\dots+q'_{n}$ are uniquely determined by $L\otimes\delta\cong\mathcal{O}_{C}(q_{1}+\dots+q_{n})$ and $\Omega_{C}^{1}\otimes{L^{\vee}}\otimes\delta^{\vee}\cong\mathcal{O}_{C}(q'_{1}+\dots+q'_{n})$.
\end{lemma}
\begin{proof}
 It is a special case of Proposition~(4.2) in \cite[Chapter~IV]{ACGH}.
\end{proof}
\begin{lemma}\label{115339_19Aug18}
 Let $K\in|\Omega_{C}^{1}|$ be an effective canonical divisor.
 If $q_{1}+\sigma(q_{1})\leq{K}$ for some $q_{1}\in{C(k)}$, then $K=\sum_{i=1}^{n}(q_{i}+\sigma(q_{i}))$ for some $q_{2},\dots,q_{n}\in{C(k)}$.
\end{lemma}
\begin{proof}
 When
 $$
 K=\sum_{i=1}^{m}(q_{i}+\sigma(q_{i}))+q+\sum_{j=1}^{2n-2m-1}r_{j}
 $$
 for $1\leq{m}\leq{n-1}$,
 we show that $\sigma(q)\in\{r_{1},\dots,r_{2n-2m-1}\}$.
 First we assume that $1\leq{m}\leq{n-2}$.
 Since $C$ is not a hyperelliptic curve, by Clifford's theorem, we have
 $$
 m+2>
 h^{0}(C,\mathcal{O}_{C}(\sum_{i=1}^{m}(q_{i}+\sigma(q_{i}))+q+\sigma(q)))
 \geq
 h^{0}(E,\mathcal{O}_{E}(\sum_{i=1}^{m}\phi(q_{i})+\phi(q)))=m+1
 $$
 and
 $$
 m+1>
 h^{0}(C,\mathcal{O}_{C}(\sum_{i=1}^{m}(q_{i}+\sigma(q_{i}))))
 \geq
 h^{0}(E,\mathcal{O}_{E}(\sum_{i=1}^{m}\phi(q_{i})))=m,
 $$
 hence $h^{0}(C,\mathcal{O}_{C}(\sum_{i=1}^{m}(q_{i}+\sigma(q_{i}))+q+\sigma(q)))={m+1}$ and $h^{0}(C,\mathcal{O}_{C}(\sum_{i=1}^{m}(q_{i}+\sigma(q_{i}))))=m$.
 Since $\sigma(q)$ is not a base point of
 $|\mathcal{O}_{C}(\sum_{i=1}^{m}(q_{i}+\sigma(q_{i}))+q+\sigma(q))|
 =\phi^{*}|\mathcal{O}_{E}(\sum_{i=1}^{m}\phi(q_{i})+\phi(q))|$,
 we have
 $$
 m=
 h^{0}(C,\mathcal{O}_{C}(\sum_{i=1}^{m}(q_{i}+\sigma(q_{i}))+q))
 <
 h^{0}(C,\mathcal{O}_{C}(\sum_{i=1}^{m}(q_{i}+\sigma(q_{i}))+q+\sigma(q)))
 =m+1,
 $$
 hence
 $$
 h^{0}(C,\mathcal{O}_{C}(\sum_{j=1}^{2n-2m-1}r_{j}-\sigma(q)))
 =h^{0}(C,\mathcal{O}_{C}(\sum_{j=1}^{2n-2m-1}r_{j}))
 =n-m-1.
 $$
 It implies that $\sigma(q)\leq\sum_{j=1}^{2n-2m-1}r_{j}$.
 We consider the case $m=n-1$.
 Let $\eta\in\Pic^{n}{(E)}$ be the invertible sheaf with $\phi^{*}\eta\cong\Omega_{C}^{1}$.
 There is a point $q'\in{C(k)}$ such that $\sum_{j=1}^{n-1}\phi(q_{i})+\phi(q')\in|\eta|$.
 Then $\mathcal{O}_{C}(q+r_{1})\cong\mathcal{O}_{C}(q'+\sigma(q'))$.
 Since $C$ is not a hyperelliptic curve, we have $q+r_{1}=q'+\sigma(q')$ and $\sigma(q)=r_{1}$.
\end{proof}
By the injective homomorphism
$$
H^{0}(E,\eta)\longrightarrow
H^{0}(C,\phi^{*}\eta)\cong
H^{0}(C,\Omega_{C}^{1}),
$$
we have the closed immersion
$$
\iota:\,\mathbf{P}(H^{0}(E,\eta)^{\vee})\longrightarrow
\mathbf{P}(H^{0}(C,\Omega_{C}^{1})^{\vee}).
$$
\begin{lemma}\label{151417_19Aug18}
 For $L\in{B_{\delta}(k)}\setminus{W_{\delta,\mathrm{sing}}(k)}$,
 $$
 \Psi_{J(C),\delta}(L)\in\iota(\mathbf{P}(H^{0}(E,\eta)^{\vee}))\subset
 \mathbf{P}(H^{0}(C,\Omega_{C}^{1})^{\vee}).
 $$
\end{lemma}
\begin{proof}
 For $L\in{B_{\delta}(k)}\setminus{W_{\delta,\mathrm{sing}}(k)}$, by Lemma~\ref{162111_7Aug18}, the image $\Psi_{J(C),\delta}(L)$ of the Gauss map is given by
 $$
 q_{1}+\dots+q_{n}+q'_{1}+\dots+q'_{n}\in|\Omega_{C}^{1}|
 \cong
 \mathbf{P}(H^{0}(C,\Omega_{C}^{1})^{\vee}),
 $$
 where the effective divisors $q_{1}+\dots+q_{n}$ and $q'_{1}+\dots+q'_{n}$ are uniquely determined by $L\otimes\delta\cong\mathcal{O}_{C}(q_{1}+\dots+q_{n})$ and $\Omega_{C}^{1}\otimes{L^{\vee}}\otimes\delta^{\vee}\cong\mathcal{O}_{C}(q'_{1}+\dots+q'_{n})$.
 Since $L\in{B_{\delta}(k)}$, by Lemma~\ref{154952_6Aug18}, $\sigma(q_{i})=q_{j}$ for some $i\neq{j}$.
 By Lemma~\ref{115339_19Aug18}, we have $\Psi_{J(C),\delta}(L)\in\iota(\mathbf{P}(H^{0}(E,\eta)^{\vee}))$.
\end{proof}
Let
$$
\Psi_{P,\delta}:D_{\delta}\setminus{D_{\delta,\mathrm{sing}}}\rightarrow
\mathbf{P}(H^{0}(P,\Omega_{P}^{1})^{\vee})
=\Grass{(n-1,H^{0}(P,\Omega_{P}^{1})^{\vee})}
$$
be the Gauss map for the subvariety $D_{\delta}\subset{P}$.
For $L\in{D_{\delta}(k)}\setminus{D_{\delta,\mathrm{sing}}(k)}$, the tangent space $T_{L}(D_{\delta})$ of $D_{\delta}$ at $L$ defines the image $\Psi_{P,\delta}(L)$ by the natural identifications
$$
T_{L}(D_{\delta})\subset
T_{L}(P)\cong(\Omega_{P}^{1}(L))^{\vee}\cong
H^{0}(P,\Omega_{P}^{1})^{\vee}.
$$
For $L\in{P(k)}$, the tangent space $T_{L}(P)$ of $P$ at $L$ is naturally identified with the orthogonal subspace
$$
V_{P}=(\phi^{*}H^{0}(E,\Omega_{E}^{1}))^{\perp}\subset
H^{0}(C,\Omega_{C}^{1})^{\vee}\cong{T_{L}(J(C))}
$$
to $\phi^{*}H^{0}(E,\Omega_{E}^{1})\subset{H^{0}(C,\Omega_{C}^{1})}$,
and it corresponds to the ramification divisor $\Ram{(\phi)}\in|\Omega_{C}^{1}|$ of the covering $\phi:C\rightarrow{E}$.
We define the finite set $\Sigma_{\delta}$ by
$$
\Sigma_{\delta}=\{\beta_{\delta}^{0}(r_{1}+\dots+r_{n})\in{J(C)}\mid
r_{1}+\dots+r_{n}\leq\Ram{(\phi)}\}.
$$
\begin{lemma}\label{165456_25Aug18}
 $D_{\delta,\mathrm{sing}}=(W_{\delta,\mathrm{sing}}\cup\Sigma_{\delta})\cap{D_{\delta}}$.
\end{lemma}
\begin{proof}
 If $L\in{D_{\delta}(k)\cap{W_{\delta,\mathrm{sing}}(k)}}$, then $L\in{D_{\delta,\mathrm{sing}}(k)}$.
 If $L\in{D_{\delta}(k)\setminus{W_{\delta,\mathrm{sing}}(k)}}$, then by Lemma~\ref{162111_7Aug18},
 $$
 L\in{D_{\delta,\mathrm{sing}}(k)}
 \Longleftrightarrow
 T_{L}(W_{\delta})=T_{L}(P)\subset{T_{L}(J(C))}
 \Longleftrightarrow
 L\in\Sigma_{\delta}.
 $$
\end{proof}
\begin{lemma}\label{153323_22Sep18}
 $(B_{\delta}\cap{P})\setminus{W_{\delta,\mathrm{sing}}}
 =(B_{\delta}\cap{P})\setminus{D_{\delta,\mathrm{sing}}}$.
\end{lemma}
\begin{proof}
 By Lemma~\ref{154952_6Aug18},
 $(B_{\delta}\setminus{W_{\delta,\mathrm{sing}}})\cap\Sigma_{\delta}
 =(B^{1}_{\delta}\setminus{W_{\delta,\mathrm{sing}}})\cap\Sigma_{\delta}$,
 and it is empty because $\Ram{(\phi)}$ is reduced.
 Hence by Lemma~\ref{165456_25Aug18},
 $$
 {W_{\delta,\mathrm{sing}}}\cap{B_{\delta}\cap{P}}
 ={W_{\delta,\mathrm{sing}}}\cap{B_{\delta}}\cap{D_{\delta}}
 =D_{\delta,\mathrm{sing}}\cap{B_{\delta}}
 =D_{\delta,\mathrm{sing}}\cap{B_{\delta}\cap{P}}.
 $$
\end{proof}
We denote by
\begin{align*}
 \pi:\,\mathbf{P}(H^{0}(C,\Omega_{C}^{1})^{\vee})\setminus\{V_{P}\}
 &\longrightarrow
 \mathbf{P}(H^{0}(P,\Omega_{P}^{1})^{\vee});\\
 [V\subset{H^{0}(C,\Omega_{C}^{1})^{\vee}}]
  &\longmapsto
  [V\cap{V_{P}}\subset{V_{P}}\cong{H^{0}(P,\Omega_{P}^{1})^{\vee}}]
\end{align*}
the projection, where $V_{P}=(\phi^{*}H^{0}(E,\Omega_{E}^{1}))^{\perp}\subset{H^{0}(C,\Omega_{C}^{1})^{\vee}}$ is the image of the dual of the restriction
$$
H^{0}(C,\Omega_{C}^{1})\cong
H^{0}(J(C),\Omega_{J(C)}^{1})
\twoheadrightarrow{H^{0}(P,\Omega_{P}^{1})}.
$$
\begin{lemma}\label{151439_19Aug18}
 $\Psi_{P,\delta}(L)=\pi\circ\Psi_{J(C),\delta}(L)$
 for $L\in{D_{\delta}(k)\setminus{D_{\delta,\mathrm{sing}}(k)}}$.
\end{lemma}
\begin{proof}
 For $L\in{D_{\delta}(k)\setminus{D_{\delta,\mathrm{sing}}(k)}}$, the tangent spaces at $L$ satisfies
 $$
 T_{L}(P)\cap{T_{L}(W_{\delta})}=T_{L}(D_{\delta})\subset{T_{L}(J(C))},
 $$
 because $P\cap{W_{\delta}}=D_{\delta}\subset{J(C)}$.
 Since $T_{L}(P)\subset{T_{L}(J(C))}$ is identified with $V_{P}\subset{H^{0}(C,\Omega_{C}^{1})^{\vee}}$ by $T_{L}(J(C))\cong{H^{0}(C,\Omega_{C}^{1})^{\vee}}$, we have $\Psi_{P,\delta}(L)=\pi\circ\Psi_{J(C),\delta}(L)$.
\end{proof}
By Lemma~\ref{151417_19Aug18}, we have the morphism
$$
\Psi^{B}_{J(C),\delta}:\,
B_{\delta}\setminus{W_{\delta,\mathrm{sing}}}\longrightarrow
\mathbf{P}(H^{0}(E,\eta)^{\vee})
$$
satisfying $\iota\circ\Psi^{B}_{J(C),\delta}=\Psi_{J(C),\delta}$.
\begin{lemma}\label{163242_24Nov18}
 The restriction $\Psi_{P,\delta}\vert_{(B_{\delta}\cap{P})\setminus{D_{\delta,\mathrm{sing}}}}$ of the Gauss map $\Psi_{P,\delta}$ is identified with the restriction $\Psi^{B}_{J(C),\delta}\vert_{(B_{\delta}\cap{P})\setminus{D_{\delta,\mathrm{sing}}}}$ of $\Psi^{B}_{J(C),\delta}$ by the isomorphism
 $$
 \pi\circ\iota:\,
 \mathbf{P}(H^{0}(E,\eta)^{\vee})\overset{\sim}{\longrightarrow}
 \mathbf{P}(H^{0}(P,\Omega_{P}^{1})^{\vee}).
 $$
\end{lemma}
\begin{proof}
 Since the composition
 $$
 H^{0}(E,\eta)\hookrightarrow
 H^{0}(C,\Omega_{C}^{1})\cong
 H^{0}(J(C),\Omega_{J(C)}^{1})
 \twoheadrightarrow{H^{0}(P,\Omega_{P}^{1})},
 $$
 is an isomorphism, it is a consequence of Lemma~\ref{151417_19Aug18} and Lemma~\ref{151439_19Aug18}.
\end{proof}
\subsection{Description for the restricted Gauss maps}
Let $\gamma_{\delta}:E^{(n-2)}\times{E}\rightarrow{J(E)}$ be the morphism defined by
\begin{align*}
 \gamma_{\delta}:\,
 E^{(n-2)}(k)\times{E(k)}&\longrightarrow\Pic^{0}{(E)};\\
 (p_{1}+\dots+p_{n-2},p)&\longmapsto
 \mathcal{O}_{E}(p_{1}+\dots+p_{n-2}+2p)\otimes(N(\delta))^{\vee}.
\end{align*}
Let $X_{\delta}\subset{E^{(n-2)}\times{E}}$ be the fiber of $\gamma_{\delta}$ at $0\in{J(E)}$, and let $Y_{\delta}\subset{C^{(n-2)}\times{E}}$ be the fiber of the composition
$$
\begin{array}{ccc}
 C^{(n-2)}\times{E}
 &\overset{\phi^{(n-2)}\times\mathrm{id}_{E}}{\longrightarrow}&
 E^{(n-2)}\times{E}
 \overset{\gamma_{\delta}}{\longrightarrow}J(E)
\end{array}
$$
at $0\in{J(E)}$.
We denote by $\psi_{\delta}:Y_{\delta}\rightarrow{X_{\delta}}$ the induced morphism by $\phi^{(n-2)}\times\mathrm{id}_{E}$.
Let $\nu_{\delta}:X_{\delta}\rightarrow|\eta|\cong\mathbf{P}(H^{0}(E,\eta)^{\vee})$ be the morphism defined by
$$
\nu_{\delta}:\,
X_{\delta}(k)\rightarrow|\eta|;\,
(p_{1}+\dots+p_{n-2},p)\longmapsto
p_{1}+\dots+p_{n-2}+p+t_{\eta-N(\delta)}(p),
$$
where $t_{\eta-N(\delta)}(p)\in{E(k)}$ is the point determined by
$$
[\mathcal{O}_{E}(t_{\eta-N(\delta)}(p))]=
[\mathcal{O}_{E}(p)]+\eta-N(\delta)\in\Pic{(E)}.
$$
We remark that $\beta^{1}_{\delta}(Y_{\delta})={B^{1}_{\delta}}\cap{P}\subset{J(C)}$, and we set
$$
Y_{\delta}^{\circ}
=(\beta^{1}_{\delta})^{-1}((B^{1}_{\delta}\cap{P})\setminus{D_{\delta,\mathrm{sing}}})
=(\beta^{1}_{\delta})^{-1}((B_{\delta}\cap{P})\setminus{D_{\delta,\mathrm{sing}}}).
$$
\begin{lemma}\label{143743_30Aug18}
 The diagram
 $$
 \begin{array}{ccc}
  Y_{\delta}^{\circ}&
   \overset{\beta^{1}_{\delta}}{\longrightarrow}&
   (B_{\delta}\cap{P})\setminus{D_{\delta,\mathrm{sing}}}\\
  \psi_{\delta}\downarrow\quad& &\quad\downarrow\Psi^{B}_{J(C),\delta}\\
  X_{\delta}&\underset{\nu_{\delta}}{\longrightarrow}&
   \mathbf{P}(H^{0}(E,\eta)^{\vee})\\
 \end{array}
 $$
 is commutative.
\end{lemma}
\begin{proof}
 Let $L\in\Pic^{0}{(C)}$ be the invertible sheaf which represents the point $\beta^{1}_{\delta}(y)\in{J(C)}$ for $y=(q_{1}+\dots+q_{n-2},\phi(q))\in{Y_{\delta}^{\circ}(k)}$.
 Then
 $q_{1}+\dots+q_{n-2}+q+\sigma(q)\in{C^{(n)}(k)}$
 is the unique effective divisor with
 $L\cong\mathcal{O}_{C}(q_{1}+\dots+q_{n-2}+q+\sigma(q))\otimes\delta^{\vee}$.
 Since
 $$
 \nu_{\delta}\circ\psi_{\delta}(y)=
 \phi(q_{1})+\dots+\phi(q_{n-2})+\phi(q)+t_{\eta-N(\delta)}(\phi(q))
 \in|\eta|,
 $$
 we have
 $$
 q_{1}+\sigma(q_{1})+\dots+q_{n-2}+\sigma(q_{n-2})+q+\sigma(q)
 +r+\sigma(r)\in|\Omega_{C}^{1}|,
 $$
 where $r\in{C(k)}$ is given by $\phi(r)=t_{\eta-N(\delta)}(\phi(q))$.
 Then $\sigma(q_{1})+\dots+\sigma(q_{n-2})+r+\sigma(r)\in{C^{(n)}(k)}$ is the unique effective divisor with $\mathcal{O}_{C}(\sigma(q_{1})+\dots+\sigma(q_{n-2})+r+\sigma(r))\cong\Omega_{C}^{1}\otimes{L^{\vee}\otimes\delta^{\vee}}$,
 and by Lemma~\ref{162111_7Aug18}, $\Phi^{B}_{J(C),\delta}(L)$ is equal to $\nu_{\delta}\circ\psi_{\delta}(y)$,
\end{proof}
\begin{lemma}\label{173307_8Sep18}
 $X_{\delta}$ and $Y_{\delta}$ are nonsingular projective varieties.
\end{lemma}
\begin{proof}
 We fix a point $p_{0}\in{E(k)}$.
 Let $\gamma_{\delta,p_{0}}:E^{(n-2)}\rightarrow{J(E)}$ be the morphism defined by
 \begin{align*}
  \gamma_{\delta,p_{0}}:\,E^{(n-2)}(k)&\longrightarrow\Pic^{0}{(E)}={J(E)(k)};\\
  p_{1}+\dots+p_{n-2}&\longmapsto
  \mathcal{O}_{E}(p_{1}+\dots+p_{n-2}+2p_{0})\otimes(N(\delta))^{\vee}.
 \end{align*}
 Then $\psi_{\delta}:Y_{\delta}\rightarrow{X_{\delta}}$ is the base change of $\phi^{(n-2)}:C^{(n-2)}\rightarrow{E^{(n-2)}}$ by the \'{e}tale covering of degree $4$;
 $$
 \begin{array}{ccccl}
  Y_{\delta}&\overset{\psi_{\delta}}{\longrightarrow}&{X_{\delta}}&
   \overset{\mathrm{pr}_{2}}{\longrightarrow}&\quad{E}\\
  \overset{\mathrm{pr}_{1}}{}\downarrow\quad&\Box&\quad\downarrow\overset{\mathrm{pr}_{1}}{}&\Box&\quad\,\downarrow(-2)_{J(E)}\circ\iota_{p_{0}}\\
  C^{(n-2)}&\underset{\phi^{(n-2)}}{\longrightarrow}&E^{(n-2)}&\underset{\gamma_{\delta,p_{0}}}{\longrightarrow}&J(E),\\
 \end{array}
 $$
 where $(-2)_{J(E)}\circ\iota_{p_{0}}:E\rightarrow{J(E)}$ is given by
 $$
 (-2)_{J(E)}\circ\iota_{p_{0}}:\,E(k)\longrightarrow\Pic^{0}{(E)}=J(E)(k);\,
 p\longmapsto\mathcal{O}_{E}(2p_{0}-2p).
 $$
\end{proof}
\begin{lemma}\label{143849_2Oct18}
 $\beta_{\delta}^{1}\vert_{Y_{\delta}^{\circ}}:Y_{\delta}^{\circ}\rightarrow
 (B_{\delta}\cap{P})\setminus{D_{\delta,\mathrm{sing}}}$
 is an isomorphism.
 In particular, $(B_{\delta}\cap{P})\setminus{D_{\delta,\mathrm{sing}}}$ is a nonsingular variety.
\end{lemma}
\begin{proof}
 By Lemma~\ref{153323_22Sep18}, the image $\beta_{\delta}^{1}(Y_{\delta}^{\circ})=(B_{\delta}\cap{P})\setminus{D_{\delta,\mathrm{sing}}}$ is a closed subset in ${W_{\delta}\setminus{W_{\delta,\mathrm{sing}}}}$.
 We show that
 $\beta_{\delta}^{1}\vert_{Y_{\delta}^{\circ}}:
 Y_{\delta}^{\circ}\rightarrow{W_{\delta}\setminus{W_{\delta,\mathrm{sing}}}}$
 is a closed immersion.
 Let $\beta^{1}:C^{(n-2)}\times{E}\rightarrow{C^{(n)}}$ be the morphism given in the proof of Lemma~\ref{171048_17Sep18}.
 Since $\beta_{\delta}^{1}=\beta_{\delta}^{0}\circ\beta^{1}$ and $\beta_{\delta}^{0}:C^{(n)}\rightarrow{J(C)}$ induces the isomorphism
 $$
 \beta_{\delta}^{0}:\,
 C^{(n)}\setminus(\beta_{\delta}^{0})^{-1}(W_{\delta,\mathrm{sing}})
 \overset{\cong}{\longrightarrow}{W_{\delta}\setminus{W_{\delta,\mathrm{sing}}}},
 $$
 it is enough to show that the finite morphism
 $$
 \beta^{1}:\,
 (C^{(n-2)}\times{E})\setminus(\beta_{\delta}^{1})^{-1}(W_{\delta,\mathrm{sing}})
 \longrightarrow
 C^{(n)}\setminus(\beta_{\delta}^{0})^{-1}(W_{\delta,\mathrm{sing}})
 $$
 is a closed immersion.
 We remark that it is injective by Lemma~\ref{171048_17Sep18}.
 For $y=(q_{1}+\dots+q_{n-2},\phi(q_{0}))\in{(C^{(n-2)}\times{E})\setminus(\beta_{\delta}^{1})^{-1}(W_{\delta,\mathrm{sing}})}$,
 we prove that the homomorphism
 $$
 T_{y}(C^{(n-2)}\times{E})\longrightarrow
 T_{\beta^{1}(y)}(C^{(n)})
 $$
 on the tangent spaces is injective.
 If $\phi(q_{0})\notin\{\phi(q_{1}),\dots,\phi(q_{n-2})\}$, then the point
 $y'=(q_{1}+\dots+q_{n-2},q_{0}+\sigma(q_{0}))\in{C^{(n-2)}(k)}\times{C^{(2)}(k)}$
 is not contained in the ramification divisor of the natural covering $C^{(n-2)}\times{C^{(2)}}\rightarrow{C^{(n)}}$.
 Since the morphism
 $$
 E(k)\rightarrow{C^{(2)}(k)};\,
 \phi(q)\longmapsto{q+\sigma(q)}
 $$ is a closed immersion,
 the homomorphism
 $$
 T_{y}(C^{(n-2)}\times{E})\hookrightarrow
 {T_{y'}(C^{(n-2)}\times{C^{(2)}})}\cong
 T_{\beta^{1}(y)}(C^{(n)})
 $$
 is injective.
 We consider the case when $y=(q_{1}+\dots+q_{n-2-i}+iq_{0},\phi(q_{0}))$ and $q_{0}\notin\{q_{1},\dots,q_{n-2-i}\}$ for some $i\geq1$.
 First we assume that $\sigma(q_{0})=q_{0}$.
 Then by Lemma~\ref{171048_17Sep18}, we have $i=1$.
 The point $\tilde{y}=(q_{1}+\dots+q_{n-3},q_{0},\phi(q_{0}))\in{C^{(n-3)}\times{C}\times{E}}$ is not contained in the ramification divisor of the natural covering
 ${C^{(n-3)}\times{C}\times{E}}\rightarrow{C^{(n-2)}\times{E}}$,
 and the point $\tilde{y}'=(q_{1}+\dots+q_{n-3},3q_{0})\in{C^{(n-3)}\times{C^{(3)}}}$ is not contained in the ramification divisor of the natural covering
 ${C^{(n-3)}\times{C^{(3)}}}\rightarrow{C^{(n)}}$.
 Since the morphism
 $$
 C(k)\times{E(k)}\rightarrow{C^{(3)}(k)};\,
 (q',\phi(q))\longmapsto{q'+q+\sigma(q)}
 $$ is a closed immersion,
 the homomorphism
 $$
 T_{y}(C^{(n-2)}\times{E})\cong
 T_{\tilde{y}}(C^{(n-3)}\times{C}\times{E})\hookrightarrow
 {T_{\tilde{y}'}(C^{(n-3)}\times{C^{(3)}})}\cong
 T_{\beta^{1}(y)}(C^{(n)})
 $$
 is injective.
 We assume that $\sigma(q_{0})\neq{q_{0}}$.
 Then by Lemma~\ref{171048_17Sep18}, we have $\sigma(q)\notin\{q_{1},\dots,q_{n-2-i}\}$.
 The point $\tilde{y}=(q_{1}+\dots+q_{n-2-i}+iq_{0},q_{0})\in{C^{(n-2)}\times{C}}$ is not contained in the ramification divisor of the covering
 $\mathrm{id}_{C^{(n-2)}}\times\phi:{C^{(n-2)}\times{C}}\rightarrow{C^{(n-2)}\times{E}}$,
 and the point $\tilde{y}'=(q_{1}+\dots+q_{n-2-i}+(i+1)q_{0},\sigma(q_{0}))\in{C^{(n-1)}\times{C}}$ is not contained in the ramification divisor of the natural covering
 ${C^{(n-1)}\times{C}}\rightarrow{C^{(n)}}$.
 Since the morphism
 \begin{align*}
  C^{(n-2)}(k)\times{C(k)}&\longrightarrow{C^{(n-1)}(k)\times{C(k)}};\\
  (q'_{1}+\dots+q'_{n-2},q)&\longmapsto{(q'_{1}+\dots+q'_{n-2}+q,\sigma(q))}
 \end{align*}
 is a closed immersion, the homomorphism
 $$
 T_{y}(C^{(n-2)}\times{E})\cong
 T_{\tilde{y}}(C^{(n-2)}\times{C})\hookrightarrow
 {T_{\tilde{y}'}(C^{(n-1)}\times{C})}\cong
 T_{\beta^{1}(y)}(C^{(n)})
 $$
 is injective.
\end{proof}
Let
$X'_{\delta}=\overline{\Psi^{B}_{J(C),\delta}
((B_{\delta}\cap{P})\setminus{D_{\delta,\mathrm{sing}}})}$
be the Zariski closure of the image of the restricted Gauss map
$\Psi^{B}_{J(C),\delta}\vert_{(B_{\delta}\cap{P})\setminus{D_{\delta,\mathrm{sing}}}}$ in $\mathbf{P}(H^{0}(E,\eta)^{\vee})$.
\begin{lemma}\label{123549_17Dec18}
 $\nu_{\delta}(X_{\delta})=X'_{\delta}$.
\end{lemma}
\begin{proof}
 By Lemma~\ref{143743_30Aug18}, we have
 $$
 \Psi^{B}_{J(C),\delta}((B_{\delta}\cap{P})\setminus{D_{\delta,\mathrm{sing}}})=\nu_{\delta}(\psi_{\delta}(Y_{\delta}^{\circ}))\subset\nu_{\delta}(X_{\delta}),
 $$
 hence
 $X'_{\delta}=
 \overline{\nu_{\delta}(\psi_{\delta}(Y_{\delta}^{\circ}))}
 \subset\nu_{\delta}(X_{\delta})$
 and
 $Y_{\delta}^{\circ}\subset(\nu_{\delta}\circ\psi_{\delta})^{-1}(X'_{\delta})$.
 Since $Y_{\delta}^{\circ}$ is dense in $Y_{\delta}$, we have
 $Y_{\delta}\subset(\nu_{\delta}\circ\psi_{\delta})^{-1}(X'_{\delta})$
 and
 $\nu_{\delta}(X_{\delta})=(\nu_{\delta}\circ\psi_{\delta})(Y_{\delta})
 \subset{X'_{\delta}}$.
\end{proof}
\begin{lemma}\label{121957_22Sep18}
 If $N(\delta)-\eta\notin{J(E)_{2}}\setminus\{0\}$, then
 $\nu_{\delta}:X_{\delta}\rightarrow{X'_{\delta}}$ is the normalization of $X'_{\delta}$.
\end{lemma}
\begin{proof}
 We set morphisms
 $\alpha_{\delta}^{\pm}:E^{(n-3)}\times{E}\rightarrow{E^{(n)}}$, $\mu_{\delta}:E^{(n-3)}\times{E}\rightarrow{E^{(n)}}$ and $\nu_{\delta}^{2}:E^{(n-4)}\times{E^{(2)}}\rightarrow{E^{(n)}}$ by
 \begin{align*}
  \alpha_{\delta}^{+}:\,
  E^{(n-3)}(k)\times{E(k)}&\longrightarrow{E^{(n)}(k)};\\
  (p_{1}+\dots+p_{n-3},p)&\longmapsto
  p_{1}+\dots+p_{n-3}+2p+t_{\eta-N(\delta)}(p),\\
  \alpha_{\delta}^{-}:\,
  E^{(n-3)}(k)\times{E(k)}&\longrightarrow{E^{(n)}(k)};\\
  (p_{1}+\dots+p_{n-3},p)&\longmapsto
  p_{1}+\dots+p_{n-3}+2p+t_{N(\delta)-\eta}(p),\\
  \mu_{\delta}:\,
  E^{(n-3)}(k)\times{E(k)}&\longrightarrow{E^{(n)}(k)};\\
  (p_{1}+\dots+p_{n-3},p)&\longmapsto
  p_{1}+\dots+p_{n-3}+p+t_{\eta-N(\delta)}(p)+t_{N(\delta)-\eta}(p)
 \end{align*}
 and
 \begin{align*}
  \nu^{2}_{\delta}:\,
 E^{(n-4)}(k)\times{E^{(2)}(k)}
 &\longrightarrow{E^{(n)}(k)};\\
 (p_{1}+\dots+p_{n-3},p+p')&\longmapsto
 p_{1}+\dots+p_{n-4}+p+t_{\eta-N(\delta)}(p)+p'+t_{\eta-N(\delta)}(p').
 \end{align*}
 By the natural inclusion
 $X'_{\delta}\subset|\eta|\subset{E^{(n)}}$,
 the subset
 $$
 U=X'_{\delta}\setminus(
 \Image{(\alpha_{\delta}^{+})}\cup\Image{(\alpha_{\delta}^{-})}\cup
 \Image{(\mu_{\delta})}\cup\Image{(\nu^{2}_{\delta})}),
 $$
 is open dense in $X'_{\delta}$, where we consider as $\Image{(\nu^{2}_{\delta})}=\emptyset$ if $n=3$.
 We show that the morphism
 $$
 \nu_{\delta}:\,
 \nu_{\delta}^{-1}(U)\longrightarrow
 E^{(n)}\setminus(
 \Image{(\alpha_{\delta}^{+})}\cup\Image{(\alpha_{\delta}^{-})}\cup
 \Image{(\mu_{\delta})}\cup\Image{(\nu^{2}_{\delta})})
 $$
 is a closed immersion.
 For $u=p_{1}+\dots+p_{n}\in{U(k)}$, we assume that
 $$
 p+t_{\eta-N(\delta)}(p)\leq
 p_{1}+\dots+p_{n}
 \quad\text{and}\quad
 p'+t_{\eta-N(\delta)}(p')\leq
 p_{1}+\dots+p_{n}
 $$
 for some $p\neq{p'}\in{E(k)}$.
 Since $u\notin\Image{(\nu_{\delta}^{2})}$, we have
 $$
 t_{\eta-N(\delta)}(p)=p'\quad\text{or}\quad
 p=t_{\eta-N(\delta)}(p'),
 $$
 and furthermore $u\notin\Image{(\mu_{\delta})}$ implies that
 $$
 t_{\eta-N(\delta)}(p)=p'\quad\text{and}\quad
 p=t_{\eta-N(\delta)}(p'),
 $$
 hence
 $N(\delta)-\eta\in{J(E)_{2}}\setminus\{0\}$.
 This means that
 $\nu_{\delta}:\nu_{\delta}^{-1}(U)\rightarrow{U}$
 is bijective if $N(\delta)-\eta\notin{J(E)_{2}}\setminus\{0\}$.
 In the following, we prove that the homomorphism
 $$
 T_{x}(E^{(n-2)}\times{E})\longrightarrow
 T_{\nu_{\delta}(x)}(E^{(n)})
 $$
 on the tangent spaces is injective for $x\in\nu_{\delta}^{-1}(U)$.
 Let $\tilde{\nu}_{\delta}:E^{(n-2)}\times{E}\rightarrow{E^{(n-2)}\times{E^{(2)}}}$ be the morphism defined by
 \begin{align*}
  \tilde{\nu}_{\delta}:\,E^{(n-2)}(k)\times{E(k)}&\longrightarrow{E^{(n-2)}(k)\times{E^{(2)}(k)}};\\
  (p_{1}+\dots+p_{n-2},p)&\longmapsto{(p_{1}+\dots+p_{n-2},p+t_{\eta-N(\delta)}(p))}.
 \end{align*}
 If $N(\delta)-\eta\notin{J(E)_{2}}\setminus\{0\}$, then the morphism $\tilde{\nu}_{\delta}$ is a closed immersion.
 For $x\in\nu_{\delta}^{-1}(U)$, the image $\tilde{\nu}_{\delta}(x)$ is not contained in the ramification divisor of the natural covering $E^{(n-2)}\times{E^{(2)}}\rightarrow{E^{(n)}}$, because $\nu_{\delta}(x)\notin\Image{(\alpha_{\delta}^{+})}\cup\Image{(\alpha_{\delta}^{-})}$.
 Hence the homomorphism
 $$
 T_{x}(E^{(n-2)}\times{E})\hookrightarrow
 T_{\tilde{\nu}_{\delta}(x)}(E^{(n)}\times{E^{(2)}})\cong
 T_{\nu_{\delta}(x)}(E^{(n)})
 $$
 is invective.
 By Lemma~\ref{173307_8Sep18}, the finite birational morphism $\nu_{\delta}:X_{\delta}\rightarrow{X'_{\delta}}$ gives the normalization of $X'_{\delta}$.
\end{proof}
\begin{remark}
 If $N(\delta)-\eta\in{J(E)_{2}}\setminus\{0\}$, then $\nu_{\delta}:X_{\delta}\rightarrow{X'_{\delta}}$ is a covering of degree $2$.
\end{remark}
\subsection{The branch locus of the restricted Gauss maps}
Let $R_{\delta}\subset{Y_{\delta}}$ be the divisor defined by
$$
R_{\delta}(k)
=\{(q_{1}+\dots+q_{n-2},p)\in{Y_{\delta}(k)}\mid
\text{$p_{i}=\sigma(p_{j})$ for some $i\neq{j}$}\}.
$$
\begin{lemma}\label{165654_18Nov18}
 $\beta_{\delta}^{1}(R_{\delta})\subset{W_{\delta,\mathrm{sing}}}$.
\end{lemma}
\begin{proof}
It is a consequence of Lemma~\ref{171048_17Sep18}, because $\beta_{\delta}^{1}(R_{\delta})\subset{B_{\delta}^{2}}$.
\end{proof}
Let $S_{\delta,r}\subset{Y_{\delta}}$ be the divisor defined by
$$
S_{\delta,r}(k)
=\{(q_{1}+\dots+q_{n-2},p)\in{Y_{\delta}(k)}\mid
q_{1}+\dots+q_{n-2}\geq{r}\}
$$
for $r\in\Ram{(\phi)}$.
Then the ramification divisor of $\psi_{\delta}:Y_{\delta}\rightarrow{X_{\delta}}$ is
$$
\Ram{(\psi_{\delta})}=R_{\delta}\cup\bigcup_{r\in\Ram{(\phi)}}S_{\delta,r}.
$$
\begin{lemma}\label{184552_17Sep18}
 $\beta_{\delta}^{1}(S_{\delta,r})\nsubseteq{W_{\delta,\mathrm{sing}}}$, and moreover $\beta_{\delta}^{1}(S_{\delta,r})\cap{W_{\delta,\mathrm{sing}}}=\emptyset$ if $n=3$.
\end{lemma}
\begin{proof}
 Let $W_{\delta,r}^{1}\subset{J(C)}$ be the subvariety defined by
 $$
 W_{\delta,r}^{1}(k)=\{L\in\Pic^{0}{(C)}\mid
 h^{0}(C,\mathcal{L}\otimes\mathcal{O}_{C}(-r)\otimes\delta)>1\},
 $$
 and let $T_{\delta,r}\subset{J(C)}$ be the image of the morphism
 \begin{align*}
  C^{(n-3)}(k)\times{E(k)}&\longrightarrow\Pic^{0}{(C)}=J(C)(k);\\
  (q_{1}+\dots+q_{n-3},\phi(q))&\longmapsto
  \mathcal{O}_{C}(q_{1}+\dots+q_{n-3}+r+q+\sigma(q))\otimes\delta^{\vee}.
 \end{align*}
 Since $C$ is not a hyperelliptic curve, by Martens' theorem~\cite[Theorem~1]{Ma}, we have $\dim{W_{\delta,r}^{1}}\leq{n-4}$ and $T_{\delta,r}\nsubseteq{W_{\delta,r}^{1}}$, hence $\dim{T_{\delta,r}}=n-2$.
 We remark that $W_{\delta,r}^{1}=\emptyset$ in the case when $n=3$.
 Since $\beta_{\delta}^{1}(S_{\delta,r})=T_{\delta,r}\cap{P}$ is the fiber of the composition
 $$
 T_{\delta,r}\subset{J(C)}\overset{N}{\rightarrow}{J(E)}
 $$
 at $0\in{J(E)}$, we have $\dim{\beta_{\delta}^{1}(S_{\delta,r})}=n-3$.
 Let $T_{\delta,2r}\subset{T_{\delta,r}}$ be the image of the morphism
 \begin{align*}
  C^{(n-4)}(k)\times{E(k)}&\longrightarrow\Pic^{0}{(C)}=J(C)(k);\\
  (q_{1}+\dots+q_{n-4},\phi(q))&\longmapsto
  \mathcal{O}_{C}(q_{1}+\dots+q_{n-4}+2r+q+\sigma(q))\otimes\delta^{\vee}.
 \end{align*}
 Since $2r=r+\sigma(r)$, we have $T_{\delta,2r}\subset{B_{\delta}^{2}}\subset{W_{\delta,\mathrm{sing}}}$ by Lemma~\ref{171048_17Sep18}.
 For
 $L\in{(T_{\delta,r}(k)\cap{W_{\delta,\mathrm{sing}}}(k))\setminus
 T_{\delta,2r}(k)}$,
 there is $(q_{1}+\dots+q_{n-3},\phi(q))\in{C^{(n-3)}(k)\times{E(k)}}$
 such that
 $L=\mathcal{O}_{C}(q_{1}+\dots+q_{n-3}+r+q+\sigma(q))
 \otimes\delta^{\vee}$
 and $r\notin\{q_{1},\dots,q_{n-3}\}$.
 If
 $$
 q'_{1}+\dots+q'_{n-i}+ir
 \in|\Omega_{C}^{1}(-q_{1}-\dots-q_{n-3}-r-q-\sigma(q))|
=|\Omega_{C}^{1}\otimes(L\otimes\delta)^{\vee}|
 $$
 and $r\notin\{q'_{1},\dots,q'_{n-i}\}$, then by Lemma~\ref{115339_19Aug18}, the number $i$ is odd.
 By the same way, any member in the linear system
 $|\Omega_{C}^{1}(-q'_{1}-\dots-q'_{n-i}-ir)|=|L\otimes\delta|$
 has an odd multiplicity at $r$.
 It implies that $h^{0}(C,L\otimes\mathcal{O}_{C}(-r)\otimes\delta)=h^{0}(C,L\otimes\delta)>1$.
 Hence we have
 $$
 T_{\delta,r}\cap{W_{\delta,\mathrm{sing}}}
 =T_{\delta,2r}\cup(T_{\delta,r}\cap{W_{\delta,r}^{1}}).
 $$
 When $n=3$, it implies that $T_{\delta,r}\cap{W_{\delta,\mathrm{sing}}}=\emptyset$.
 When $n\geq4$,
 $$
 \beta_{\delta}^{1}(S_{\delta,r})\cap{W_{\delta,\mathrm{sing}}}
 =(T_{\delta,2r}\cap{P})\cup(\beta_{\delta}^{1}(S_{\delta,r})
 \cap{W_{\delta,r}^{1}})
 $$
 is a proper closed subset of
 $\beta_{\delta}^{1}(S_{\delta,r})=T_{\delta,r}\cap{P}$, because $\dim{(T_{\delta,2r}\cap{P})}\leq{n-4}$ and $\dim{W_{\delta,r}^{1}}\leq{n-4}$.
\end{proof}
Let $Z_{\delta,r}=\psi_{\delta}(S_{\delta,r})$ be the image of $S_{\delta,r}$ by $\psi_{\delta}:Y_{\delta}\rightarrow{X_{\delta}}$.
Then
$$
Z_{\delta,r}(k)=\{(p_{1}+\dots+p_{n-2},p)\in{X_{\delta}}\mid
\phi(r)\leq{p_{1}+\dots+p_{n-2}}\}.
$$
\begin{lemma}\label{175903_15Sep18}
 If $n\geq4$, then $Z_{\delta,r}$ is irreducible.
 If $n=3$, then $X_{\delta}\cong{E}$, and $Z_{\delta,r}\subset{X_{\delta}}$ is a $J(X_{\delta})_{2}$-orbit by the natural action of $J(X_{\delta})$ on the curve $X_{\delta}$ of genus $1$.
\end{lemma}
\begin{proof}
 Let $\gamma_{\delta,p_{0}}:E^{(n-2)}\rightarrow{J(E)}$ be the morphism given in the proof of Lemma~\ref{173307_8Sep18} for fixed $p_{0}\in{E(k)}$, and let $i_{r}:E^{(n-3)}\rightarrow{E^{(n-2)}}$ be the morphism defined by
 $$
 i_{r}:\,E^{(n-3)}(k)\longrightarrow{E^{(n-2)}(k)};\,
 p_{1}+\dots+p_{n-3}\longmapsto{p_{1}+\dots+p_{n-3}+\phi(r)}.
 $$
 If $n\geq4$, then $Z_{\delta,r}$ is a $\mathbf{P}^{n-4}$-bundle over $E$ by the base change
 $$
 \begin{array}{ccccl}
  Z_{\delta,r}&{\longrightarrow}&{X_{\delta}}&
   \overset{\mathrm{pr}_{2}}{\longrightarrow}&\quad{E}\\
  \downarrow&\Box&\quad\downarrow\overset{\mathrm{pr}_{1}}{}&\Box&\quad\,\downarrow(-2)_{J(E)}\circ\iota_{p_{0}}\\
  E^{(n-3)}&\underset{i_{r}}{\longrightarrow}&E^{(n-2)}&\underset{\gamma_{\delta,p_{0}}}{\longrightarrow}&J(E)\\
 \end{array}
 $$
 of the $\mathbf{P}^{n-4}$-bundle
 $\gamma_{\delta,p_{0}}\circ{i_{r}}:E^{(n-3)}\rightarrow{J(E)}$, hence $Z_{\delta,r}$ is irreducible.
 If $n=3$, then $\mathrm{pr}_{2}:X_{\delta}{\rightarrow}{E}$ is an isomorphism, and
 $$
 Z_{\delta,r}\cong\{p\in{E(k)}\mid
 {\mathcal{O}_{E}(\phi(r)+2p)}\cong{N(\delta)}\}
 $$
 is an orbit of $J(E)_{2}$-action.
\end{proof}
We denote by $\Ram{(\psi_{\delta}^{\circ})}\subset{Y_{\delta}^{\circ}}$ the ramification divisor of $\psi_{\delta}^{\circ}=\psi_{\delta}\vert_{Y_{\delta}^{\circ}}:Y_{\delta}^{\circ}\rightarrow{X_{\delta}}$.
\begin{lemma}\label{113332_24Nov18}
 $$
 \overline{\psi_{\delta}^{\circ}(\Ram{(\psi_{\delta}^{\circ})})}
 =\bigcup_{r\in\Ram{(\phi)}}Z_{\delta,r}.
 $$
\end{lemma}
\begin{proof}
 Since $\Ram{(\psi_{\delta})}={R_{\delta}}\cup\bigcup_{r\in\Ram{(\phi)}}S_{\delta,r}$, by Lemma~\ref{165654_18Nov18}, we have $\Ram{(\psi_{\delta}^{\circ})}=\bigcup_{r\in\Ram{(\phi)}}S_{\delta,r}\cap{Y_{\delta}^{\circ}}$.
 By Lemma~\ref{184552_17Sep18}, $S_{\delta,r}\cap{Y_{\delta}^{\circ}}\neq\emptyset$ for $n\geq3$, and $S_{\delta,r}\cap{Y_{\delta}^{\circ}}=S_{\delta,r}$ for $n=3$.
 Since $\psi_{\delta}$ is a finite morphism, $\psi_{\delta}^{\circ}(S_{\delta,r}\cap{Y_{\delta}^{\circ}})$ is of dimension $n-3$.
 By Lemma~\ref{175903_15Sep18}, we have
 $\overline{\psi_{\delta}^{\circ}(S_{\delta,r}\cap{Y_{\delta}^{\circ}})}=\psi_{\delta}(S_{\delta,r})=Z_{\delta,r}$.
\end{proof}
Let $H_{r}\subset\mathbf{P}(H^{0}(E,\eta)^{\vee})$ be the hyperplane corresponding to the subspace
$$
H^{0}(E,\eta\otimes\mathcal{O}_{E}(-\phi(r)))
\subset{H^{0}(E,\eta)}.
$$
\begin{lemma}\label{175851_15Sep18}
 $H_{r}$ is the unique hyperplane with the property $\nu_{\delta}(Z_{\delta,r})\subset{H_{r}}$.
\end{lemma}
\begin{proof}
 The inclusion $\nu_{\delta}(Z_{\delta,r})\subset{H_{r}}$ is obvious.
 We prove the uniqueness of the hyperplane $H_{r}$.
 Let $z=(p_{1}+\dots+p_{n-3}+\phi(r),p)\in{Z_{\delta,r}(k)}$ be satisfying $p_{i}\neq{t_{\eta-N(\delta)}(p_{j})}$ for $i\neq{j}$.
 We take a point $p'\in{E(k)}\setminus\{p\}$ such that $\mathcal{O}_{E}(2p')\cong\mathcal{O}_{E}(2p)$ and $\mathcal{O}_{E}(p'-p)\ncong{\eta\otimes{N(\delta)^{\vee}}}$.
 Then $z'=(p_{1}+\dots+p_{n-3}+\phi(r),p')$ is contained in ${Z_{\delta,r}(k)}$, and $\nu_{\delta}(z)\neq\nu_{\delta}(z')$.
 It implies the uniqueness in the case when $n=3$.
 When $n\geq4$, we show that $\nu_{\delta}(Z_{\delta,r})\subset{H_{r}}$ is a non-linear hypersurface in $H_{r}$.
 Let $l\subset{H_{r}}$ be the line containing the two points $\nu_{\delta}(z),\nu_{\delta}(z')\in{\nu_{\delta}(Z_{\delta,r})}$.
 Then the line $l\subset\mathbf{P}(H^{0}(E,\eta)^{\vee})$ corresponds to the linear pencil
 $$
 |\eta(-p_{1}-\dots-p_{n-3}-\phi(r))|\subset|\eta|
 \cong\mathbf{P}(H^{0}(E,\eta)^{\vee}).
 $$
 For a point $p_{0}\in{E(k)}$, there is a unique point $p'_{0}\in{E(k)}$ such that $p_{0}+p'_{0}\in|\eta(-p_{1}-\dots-p_{n-3}-\phi(r))|$.
 If $\mathcal{O}_{E}(2p_{0})\ncong\mathcal{O}_{E}(2p)$, $\mathcal{O}_{E}(2p'_{0})\ncong\mathcal{O}_{E}(2p)$ and
 $$
 p_{0},p'_{0}\notin\{
 t_{\eta-N(\delta)}(p_{1}),\dots,t_{\eta-N(\delta)}(p_{n-3}),
 t_{N(\delta)-\eta}(p_{1}),\dots,t_{N(\delta)-\eta}(p_{n-3})\},
 $$
 then the point $p_{1}+\dots+p_{n-3}+\phi(r)+p_{0}+p'_{0}\in|\eta|$ on the line $l$ is not contained in $\nu_{\delta}(Z_{\delta,r})$.
\end{proof}
\begin{lemma}\label{183745_14Dec18}
 The pull-back of the divisor $H_{r}$ by $\nu_{\delta}:X_{\delta}\rightarrow\mathbf{P}(H^{0}(E,\eta)^{\vee})$ is
 $$
 \nu_{\delta}^{*}{H_{r}}=Z_{\delta,r}+{M_{\delta,r}}+{M'_{\delta,r}},
 $$
 where $M_{\delta,y}$ and $M'_{\delta,y}$ are irreducible divisors on $X_{\delta}$ defined by
 \begin{align*}
  &M_{\delta,r}(k)=\{(p_{1}+\dots+p_{n-2},p)\in{X_{\delta}(k)}\mid
   p=\phi(r)\},\\
  &M'_{\delta,r}(k)=\{(p_{1}+\dots+p_{n-2},p)
   \in{X_{\delta}(k)}\mid
   p=t_{N(\delta)-\eta}(\phi(r))\}.
 \end{align*}
\end{lemma}
\begin{proof}
 Let $I_{r}$ be an irreducible divisor on $E^{(n)}$ defined by
 $$
 I_{r}(k)=\{p_{1}+\dots+p_{n}\in{E^{(n)}(k)}\mid
 p_{1}+\dots+p_{n}\geq\phi(r)\},
 $$
 and let $Z_{r}$, $M_{r}$, $M'_{r}$ be irreducible divisors on $E^{(n-2)}\times{E}$ defined by
 \begin{align*}
  &Z_{r}(k)=\{(p_{1}+\dots+p_{n-2},p)\in{E^{(n-2)}(k)\times{E(k)}}\mid
   p_{1}+\dots+p_{n-2}\geq\phi(r)\},\\
  &M_{r}(k)=\{(p_{1}+\dots+p_{n-2},p)\in{E^{(n-2)}(k)\times{E(k)}}\mid
   p=\phi(r)\},\\
  &M'_{r}(k)=\{(p_{1}+\dots+p_{n-2},p)
   \in{E^{(n-2)}(k)\times{E(k)}}\mid
   p=t_{N(\delta)-\eta}(\phi(r))\}.
 \end{align*}
 Then the pull-back of the divisor $I_{r}$ by the morphism
 \begin{align*}
  E^{(n-2)}(k)\times{E(k)}&\longrightarrow{E^{(n)}(k)};\\
  (p_{1}+\dots+p_{n-2},p)&\longmapsto
  p_{1}+\dots+p_{n-2}+p+t_{\eta-N(\delta)}(p)
 \end{align*}
 is the divisor $Z_{r}+M_{r}+M'_{r}$ on $E^{(n-2)}\times{E}$.
 Since the restriction of $I_{r}$ to $|\eta|\subset{E^{(n)}}$ is the divisor $H_{r}$ on $\mathbf{P}(H^{0}(E,\eta)^{\vee})\cong|\eta|$, the pull-back $\nu_{\delta}^{*}H_{r}$ is the restriction of $Z_{r}+M_{r}+M'_{r}$ to $X_{\delta}$.
\end{proof}
\begin{corollary}\label{143241_6Oct18}
 $\nu_{\delta}^{*}{H_{r}}-Z_{\delta,r}$ is an irreducible divisor on $X_{\delta}$ if and only if $N(\delta)=\eta$.
\end{corollary}
We consider the dual variety $(\Phi_{|\eta|}(E))^{\vee}\subset\mathbf{P}(H^{0}(E,\eta)^{\vee})$ of the image of the closed immersion $\Phi_{|\eta|}:E\rightarrow\mathbf{P}(H^{0}(E,\eta))$.
\begin{lemma}\label{152431_2Nov18}
 The projective curve $\Phi_{|\eta|}(E)\subset\mathbf{P}(H^{0}(E,\eta))$ is reflexive.
 In particular, $\Phi_{|\eta|}(E)=((\Phi_{|\eta|}(E))^{\vee})^{\vee}\subset\mathbf{P}(H^{0}(E,\eta))$.
\end{lemma}
\begin{proof}
 If $1\leq{i}<n$, then $h^{0}(E,\eta\otimes\mathcal{O}_{E}(-ip))=n-i$ for any $p\in{E(k)}$.
 If $n=3$, then $h^{0}(E,\eta\otimes\mathcal{O}_{E}(-3p))=0$ for general $p\in{E(k)}$.
 Hence $h^{0}(E,\eta\otimes\mathcal{O}_{E}(-2p))>h^{0}(E,\eta\otimes\mathcal{O}_{E}(-3p))$ for general $p\in{E(k)}$.
 Then there is a hyperplane $H\subset\mathbf{P}(H^{0}(E,\eta))$ which intersects $\Phi_{|\eta|}(E)$ at $\Phi_{|\eta|}(p)$ with the multiplicity $2$.
 By \cite[(3.5)]{HK}, $\Phi_{|\eta|}(E)\subset\mathbf{P}(H^{0}(E,\eta))$ is reflexive, because the characteristic of the base field $k$ is not equal to $2$.
\end{proof}
\begin{lemma}\label{143345_6Oct18}
 If $N(\delta)=\eta$, then the dual variety of $X'_{\delta}\subset\mathbf{P}(H^{0}(E,\eta)^{\vee})$ is
 $\Phi_{|\eta|}(E)\subset\mathbf{P}(H^{0}(E,\eta))$.
\end{lemma}
\begin{proof}
 By Lemma~\ref{152431_2Nov18}, we show that the dual variety of $\Phi_{|\eta|}(E)$ is $X'_{\delta}$.
 For
 $L\in{(B_{\delta}\cap{P})\setminus{D_{\delta,\mathrm{sing}}}}
 \subset\Pic^{0}{(C)}$,
 there is a unique effective divisor $q_{1}+\dots+q_{n-2}+q+\sigma(q)\in{C^{(n)}(k)}$ such that
 $$
 L\otimes\delta\cong\mathcal{O}_{C}(q_{1}+\dots+q_{n-2}+q+\sigma(q)).
 $$
 Since $L\in{P(k)}$, we have
 $$
 \eta=N(\delta)=[\mathcal{O}_{E}(\phi(q_{1})+\dots+\phi(q_{n-2})+2\phi(q))],
 $$
 hence
 $$
 \Omega_{C}^{1}\otimes{L^{\vee}}\otimes\delta^{\vee}\cong
 \phi^{*}\eta\otimes{L^{\vee}}\otimes\delta^{\vee}\cong
 \mathcal{O}_{C}(\sigma(q_{1})+\dots+\sigma(q_{n-2})+\sigma(q)+q)
 $$
 and $\Psi^{B}_{J(C),\delta}(L)\in\mathbf{P}(H^{0}(E,\eta)^{\vee})$ is defined by the effective divisor
 $$
 \phi(q_{1})+\dots+\phi(q_{n-2})+2\phi(q)\in|\eta|
 \cong\mathbf{P}(H^{0}(E,\eta)^{\vee}).
 $$
 It means that the hyperplane in $\mathbf{P}(H^{0}(E,\eta))$ corresponding $\Psi^{B}_{J(C),\delta}(L)$ is tangent to the image $\Phi_{|\eta|}(E)$.
 Hence we have
 $$
 \Psi^{B}_{J(C),\delta}((B_{\delta}\cap{P})\setminus{D_{\delta,\mathrm{sing}}})
 \subset(\Phi_{|\eta|}(E))^{\vee}.
 $$
 Since $(\Phi_{|\eta|}(E))^{\vee}$ and $\Psi^{B}_{J(C),\delta}((B_{\delta}\cap{P})\setminus{D_{\delta,\mathrm{sing}}})$ are irreducible hypersurfaces in $\mathbf{P}(H^{0}(E,\eta)^{\vee})$, we have $X'_{\delta}=(\Phi_{|\eta|}(E))^{\vee}$.
\end{proof}
\section{Key Propositions}\label{150759_29Nov18}
Let $\mathcal{L}$ be an ample invertible sheaf on $P$ which represents the the polarization $\lambda_{P}$.
\begin{lemma}\label{152342_24Nov18}
 $U_{D}=\Bs|\mathcal{L}|\setminus{D_{\mathrm{sing}}}$ is nonsingular for any $D\in|\mathcal{L}|$.
\end{lemma}
\begin{proof}
 Since $D=D_{\delta}$ for some $\delta\in\Pic^{n}{(C)}$, it is a consequence of Lemma~\ref{135618_16Aug18} and Lemma~\ref{143849_2Oct18}.
\end{proof}
Let
$$
\Psi_{D}:D\setminus{D_{\mathrm{sing}}}\longrightarrow
\mathbf{P}^{n-1}=\Grass{(n-1,H^{0}(P,\Omega_{P}^{1})^{\vee})}
$$
be the Gauss map for $D\in|\mathcal{L}|$, and let $\nu_{D}:X_{D}\rightarrow{X'_{D}}$ be the normalization of $X'_{D}=\overline{\Psi_{D}(U_{D})}\subset\mathbf{P}^{n-1}$.
Then by Lemma~\ref{152342_24Nov18}, there is a unique morphism $\psi_{D}:U_{D}\rightarrow{X_{D}}$ such that $\Psi_{D}\vert_{U_{D}}=\nu_{D}\circ\psi_{D}$.
 Let $Z_{D}=\overline{\psi_{D}(\Ram{(\psi_{D})})}\subset{X_{D}}$ be the Zariski closure of the image of the ramification divisor of $\psi_{D}$.
\begin{proposition}\label{165452_24Nov18}
 Let $D\subset{P}$ be a member in $|\mathcal{L}|\setminus\Pi_{\mathcal{L}}$, where $\Pi_{\mathcal{L}}\subset|\mathcal{L}|$ is the subset in Lemma~\ref{142326_6Oct18}.
\begin{enumerate}\label{174153_24Nov18}
 \item If $n=3$, then $X_{D}$ is a nonsingular projective curve of genus $1$, and $Z_{D}$ is a disjoint union of $6$ orbits $Z_{D,1},\dots,Z_{D,6}$ by the $J(X_{D})_{2}$-action.
 \item If $n\geq4$, then $Z_{D}$ has $2n$ irreducible components $Z_{D,1},\dots,Z_{D,2n}$.
 \item For any subset $Z_{D,i}\subset{Z_{D}}$ in $(1)$ and $(2)$, there is a unique hyperplane $H_{D,i}\subset\mathbf{P}^{n-1}$ such that $\nu_{D}(Z_{D,i})\subset{H_{D,i}}$.
\end{enumerate}
\end{proposition}
\begin{proof}
 By Lemma~\ref{135426_16Aug18}, there is $\delta\in\Pic^{n}{(C)}$ such that $N(\delta)=\eta$ and $\mathcal{L}\cong\mathcal{L}_{\delta}$.
 By Lemma~\ref{135618_16Aug18}, there is $s\in\Pic^{0}{(E)}$ such that $D=D_{\delta+\phi^{*}s}$.
 By the proof of Lemma~\ref{142326_6Oct18}, $D\notin\Pi_{\mathcal{L}}$ implies $s\notin{J(E)_{4}}\setminus{J(E)_{2}}$.
 By Lemma~\ref{163242_24Nov18}, the Gauss map $\Psi_{D}\vert_{U_{D}}:U_{D}\rightarrow\mathbf{P}^{n-1}$ is identified with $\Psi^{B}_{J(C),\delta+\phi^{*}s}\vert_{U_{D}}:U_{D}\rightarrow\mathbf{P}(H^{0}(E,\eta)^{\vee})$.
 Since
 $N(\delta+\phi^{*}s)-\eta
 =2s\notin{J(E)_{2}}\setminus\{0\}$,
 by Lemma~\ref{123549_17Dec18} and Lemma~\ref{121957_22Sep18}, the normalization of $X'_{\delta+\phi^{*}s}={X'_{D}}$ is given by $\nu_{\delta+\phi^{*}s}:X_{\delta+\phi^{*}s}\rightarrow{X'_{\delta+\phi^{*}s}}$, and by Lemma~\ref{143743_30Aug18} and Lemma~\ref{143849_2Oct18}, $\psi_{D}:U_{D}\rightarrow{X_{D}}$ is identified with $\psi^{\circ}_{\delta+\phi^{*}s}:Y^{\circ}_{\delta+\phi^{*}s}\rightarrow{X_{\delta+\phi^{*}s}}$.
 Hence the statements (1), (2) and (3) are consequence of Lemma~\ref{175903_15Sep18}, Lemma~\ref{113332_24Nov18} and Lemma~\ref{175851_15Sep18}.
\end{proof}
We define the subset $\Pi'_{\mathcal{L}}$ in the linear pencil $|\mathcal{L}|$ by
$$
\Pi'_{\mathcal{L}}=\{D\in|\mathcal{L}|\setminus\Pi_{\mathcal{L}}\mid\text{
$\nu_{D}^{*}{H_{D,i}}-Z_{D,i}$ is irreducible for $1\leq{i}\leq{2n}$}\}.
$$
\begin{lemma}\label{134511_15Dec18}
 $\sharp\Pi'_{\mathcal{L}}=4$.
\end{lemma}
\begin{proof}
 We use the same identification for Gauss maps as in the proof of Proposition~\ref{165452_24Nov18}.
 Then by Corollary~\ref{143241_6Oct18},
 $$
 D=D_{\delta+\phi^{*}s}\in\Pi'_{\mathcal{L}}
 \Longleftrightarrow
 N(\delta+\phi^{*}s)=\eta
 \Longleftrightarrow
 s\in{J(E)_{2}},
 $$
 and by Lemma~\ref{184538_11Aug18}, we have $\sharp\Pi'_{\mathcal{L}}=\sharp{J(C)_{2}}=4$.
\end{proof}
Let $e_{1}+\dots+e_{2n}$ be the branch divisor of the original covering $\phi:C\rightarrow{E}$, and let $\eta\in\Pic{(E)}$ be the invertible sheaf with $\phi^{*}\eta\cong\Omega_{C}^{1}$.
\begin{proposition}\label{173842_24Nov18}
 For any member $D\in\Pi'_{\mathcal{L}}$, there is an isomorphism
 $$
 (E,e_{1}+\dots+e_{2n},\eta)\cong
 ((X'_{D})^{\vee},H_{D,1}^{\vee}+\dots+H_{D,2n}^{\vee},\mathcal{O}_{(\mathbf{P}^{n-1})^{\vee}}(1)\vert_{(X'_{D})^{\vee}}),
 $$
 where $H_{D,i}^{\vee}\in(\mathbf{P}^{n-1})^{\vee}$ is the point corresponding to the hyperplane $H_{D,i}$, and $(X'_{D})^{\vee}\subset(\mathbf{P}^{n-1})^{\vee}$ is the dual variety of $X'_{D}\subset\mathbf{P}^{n-1}$.
\end{proposition}
\begin{proof}
 We use the same identification for Gauss maps as in the proof of Proposition~\ref{165452_24Nov18}.
 When $D\in\Pi'_{\mathcal{L}}$, we may assume that $D=D_{\delta}$ and $N(\delta)=\eta$ by Corollary~\ref{143241_6Oct18}.
 Then the point $H_{D,i}^{\vee}$ is identified with the point $H_{r}^{\vee}=\Phi_{|\eta|}(\phi(r))$ for $r\in\Ram{(\phi)}$, and $(X'_{D})^{\vee}$ is identified with $(X'_{\delta})^{\vee}$, which coincides with
$\Phi_{|\eta|}(E)\subset\mathbf{P}(H^{0}(E,\eta))$ by Lemma~\ref{143345_6Oct18}.
\end{proof}
\begin{remark}
 For a member $D\in\Pi'_{\mathcal{L}}$ the Gauss map $\Psi_{D}:D\setminus{D_{\mathrm{sing}}}\rightarrow\mathbf{P}^{n-1}$ is of degree $2^{n}$, and $X'_{D}+\sum_{i=1}^{2n}H_{D,i}$ is the branch divisor of $\Psi_{D}$.
 But for $D\notin\Pi'_{\mathcal{L}}$ the Gauss map $\Psi_{D}$ is not easy to compute.
\end{remark}
\subsection*{Acknowledgments}
The author would like to thank Juan Carlos Naranjo for the information on the recent paper \cite{NO}.

\end{document}